\documentclass[12pt, reqno]{amsart}
\usepackage{amsmath, amsthm, amscd, amsfonts, amssymb, graphicx, color}
\usepackage{mathrsfs}
\usepackage{newtxtext}
\usepackage[bookmarksnumbered, colorlinks, plainpages]{hyperref}
\hypersetup{colorlinks=true,linkcolor=red, anchorcolor=green, citecolor=cyan, urlcolor=red, filecolor=magenta, pdftoolbar=true}

\newtheorem{theorem}{Theorem}[section]
\newtheorem{lemma}[theorem]{Lemma}

\newtheorem{corollary}[theorem]{Corollary}

\theoremstyle{remark}
\newtheorem{remark}[theorem]{\bf{Remark}}
\numberwithin{equation}{section}
\allowdisplaybreaks
\begin{document}

\title [Estimation of $\mathrm{A}$-Berezin number and  $\mathrm{A}$-Berezin norm inequalities \ldots]{Estimation of $\mathrm{A}$-Berezin number and  $\mathrm{A}$-Berezin norm inequalities via Moore-Penrose inverse}
\author[S. Ghosh, S. Saha Mondal, S. Ojha, ]{ Sumon Ghosh, Swastika Saha Mondal, Sarita Ojha }

\address[Ghosh]{Department of Mathematics, Indian Institute of Engineering Science and Technology, Shibpur, 711103, West Bengal, India}
\email{isumonghoshmath@gmail.com}

\address[Saha Mondal]{Department of Mathematics, Raghu Engineering College (Autonomous), Visakhapatnam, Andhra Pradesh- 531162, India}
\email{swastika.sm95@gmail.com}

\address[Ojha]{Department of Mathematics, Indian Institute of Engineering Science and Technology, Shibpur, 711103, West Bengal, India}
\email{sarita.ojha89@gmail.com}

\subjclass[2020]{47A05, 47B15, 47B32}

\keywords{Berezin number; Berezin norm; Reproducing kernel Hilbert space;
Moore-Penrose inverse}

\begin{abstract}
In this article, we establish the $A$-Berezin number and $A$-Berezin norm inequalities for bounded linear operators on a reproducing kernel Hilbert space using the Moore-Penrose inverse. We further extend these inequalities to the case of the sum of two bounded linear operators. As an application, upper bounds of the Berezin number and Berezin norm of block matrices have been discussed via the Moore-Penrose inverse. The inequalities established here offer both refinements and generalizations of previous results.
\end{abstract}

\maketitle

\section{Introduction}

Let $\mathcal{B(H)}$ be the $C^*$-algebra of all bounded linear operators on a complex Hilbert space $\mathcal{H}$ with inner product $\left\langle\cdot,\cdot\right\rangle$ and associated norm $\|\cdot\|$. The usual operator norm of $\mathrm{T}$ is defined as $\|\mathrm{T}\|=\sup\{|\langle \mathrm{T}u,v\rangle|:u,v\in\mathcal{H},~\|u\|=\|v\|=1\}$. Let $\mathcal{B(H)}^+$ be the set of all positive operators of $\mathcal{B(H)}$, i.e., $\mathcal{B(H)}^+=\{\mathrm{T}\in\mathcal{B(H)}:\langle \mathrm{T}u,u\rangle\geq 0 \  ~\forall~ u\in\mathcal{H}\}$. For $\mathrm{T}\in\mathcal{B(H)}$, $\mathrm{T}^*$ denotes the adjoint of $\mathrm{T}$, and $|\mathrm{T}|$ denotes the positive operator $\sqrt{\mathrm{T}^*\mathrm{T}}$. Throughout the article, the null space of an operator $\mathrm{T}$ is denoted by $\mathcal{N}(\mathrm{T})$, and its range by $\mathcal{R}(\mathrm{T})$. $\overline{\mathcal{R}(\mathrm{T})}$ is the closure of $\mathcal{R}(\mathrm{T})$ with respect to the usual norm of $\mathcal{H}$. 

For a given $\mathrm{A}\in \mathcal{B(H)}^+$, consider a semi-inner product $\langle \cdot,\cdot \rangle_\mathrm{A}$ on $\mathcal{H}$ defined by $\langle u,v \rangle_\mathrm{A} =\langle \mathrm{A}u,v \rangle, ~\mbox{for all}~ u,v\in\mathcal{H}$. Clearly, the induced seminorm is given by $\|u\|_\mathrm{A}=\langle u,u \rangle_\mathrm{A}^\frac{1}{2}=\|\mathrm{A}^\frac{1}{2}u\|$ for every $u\in\mathcal{H}$. The vector space $\mathcal{H}$  endowed with the semi-inner product $\langle \cdot,\cdot\rangle_\mathrm{A}$ is called a semi-Hilbertian space. It can be shown that $\|\cdot\|_\mathrm{A}$ is a norm on $\mathcal{H}$ if and only if $\mathrm{A}$ is one-to-one, and the semi-Hilbertian space $(\mathcal{H},\|\cdot\|_\mathrm{A})$ is complete if and only if $\mathcal{R}(\mathrm{A})=\overline{\mathcal{R}(\mathrm{A})}$. For $\mathrm{T}\in\mathcal{B(H)}$, if there exists a constant $c>0$ such that $\|\mathrm{T}u\|_\mathrm{A}\leq c\|u\|_\mathrm{A}$ for all $u\in\overline{\mathcal{R}(\mathrm{A})}$, then the $\mathrm{A}$-operator semi-norm of $\mathrm{T}$, denoted by $\|\mathrm{T}\|_\mathrm{A}$, is defined as
\begin{equation*}
\|\mathrm{T}\|_\mathrm{A}=\underset{u\in\overline{\mathcal{R}(\mathrm{A})},\ u\neq 0}{\sup}\frac{\|\mathrm{T}u\|_\mathrm{A}}{\|u\|_\mathrm{A}}.  
\end{equation*}
For an operator $\mathrm{T}\in\mathcal{B(H)}$, an operator $\mathrm{S}\in\mathcal{B(H)}$ is referred to as an $\mathrm{A}$-adjoint of $\mathrm{T}$ if $\langle \mathrm{T}u,v \rangle_\mathrm{A}=\langle u,\mathrm{S}v \rangle_\mathrm{A}$ for all $u,v\in\mathcal{H}$. Generally, the existence of an $\mathrm{A}$-adjoint operator is not guaranteed. By Douglas Theorem \cite{Douglas-majorization}, the set of all operators in $\mathcal{B(H)}$, which admit an $\mathrm{A}$-adjoint is denoted by $\mathcal{B}_\mathrm{A}\mathcal{(H)}$, and is defined by
\begin{equation*}
 \mathcal{B}_\mathrm{A}\mathcal{(H)}=\{\mathrm{T}\in\mathcal{B}\mathcal{(H)}: \mathcal{R}(\mathrm{T}^*\mathrm{A}) \subseteq \mathcal{R}(\mathrm{A})\}.   
\end{equation*}
If $\mathrm{T}\in\mathcal{B}_\mathrm{A}\mathcal{(H)}$, then by Douglas Theorem, the equation $\mathrm{AY}=\mathrm{T}^*\mathrm{A}$ has a unique solution, denoted by $\mathrm{T}^{\#_\mathrm{A}}$, which satisfies $\mathcal{R}(\mathrm{T}^{\#_\mathrm{A}})\subseteq \overline{\mathcal{R}(\mathrm{A})}$. The set $\mathcal{B}_{\mathrm{A}^\frac{1}{2}}\mathcal{(H)}$ represents all operators that admit an $\mathrm{A}^\frac{1}{2}$-adjoint. Again, by Douglas Theorem, this set can be characterized as
\begin{equation*}
  \mathcal{B}_{\mathrm{A}^\frac{1}{2}}\mathcal{(H)}=\{\mathrm{T}\in\mathcal{B(H)}:~ \exists~\lambda>0~\mbox{such that}~\|\mathrm{T}u\|_\mathrm{A}\leq\lambda\|u\|_\mathrm{A}~\forall ~u\in\mathcal{H}\}.  
\end{equation*}
An operator $\mathrm{T}\in\mathcal{B}_{\mathrm{A}^\frac{1}{2}}\mathcal{(H)}$ is known as $\mathrm{A}$-bounded operator. The inclusions $\mathcal{B}_\mathrm{A}\mathcal{(H)}\subseteq \mathcal{B}_{\mathrm{A}^\frac{1}{2}}\mathcal{(H)}\subseteq\mathcal{B}\mathcal{(H)}$
hold with equality if $\mathrm{A}$ is one-to-one and has closed range (see \cite{Arias-Metric,Arias-Partial,Feki-Spectral}).

Let $\mathcal{CR(H)}$ be the set defined by $\mathcal{CR(H)}=\{\mathrm{T}\in\mathcal{B(H)}:\mathcal{R}(\mathrm{T}) ~\text{is closed}\}$. If $\mathrm{T}\in \mathcal{CR(H)}$, then there exists a unique $\mathrm{T}^\dag\in\mathcal{B(H)}$ which satisfies the following relations:
\begin{equation*}
    (\mbox{a})~\mathrm{TT}^\dag \mathrm{T}=\mathrm{T},~(\mbox{b})~\mathrm{T}^\dag \mathrm{TT}^\dag=\mathrm{T}^\dag,~(\mbox{c})~(\mathrm{TT}^\dag)^*=\mathrm{TT}^\dag,~(\mbox{d})~(\mathrm{T}^\dag \mathrm{T})^*=\mathrm{T}^\dag \mathrm{T}.
\end{equation*}
 Here, the operator $\mathrm{T}^\dag$ is called the Moore-Penrose inverse of $\mathrm{T}$. For a detailed study of Moore-Penrose inverse, we refer to \cite{ben2002moore,moore1920reciprocal,penrose1955generalized}. It is well-known that $\mathrm{T}^{\#_\mathrm{A}}=\mathrm{A}^\dag \mathrm{T}^* \mathrm{A}$. 

Let $\mathrm{A}\in\mathcal{B(H)}^+$. An operator $\mathrm{T}^{\dag_\mathrm{A}}\in \mathcal{B(H)}$ is called an $\mathrm{A}$-generalized inverse (see \cite{Arias-A-Partial}) of $\mathrm{T}\in\mathcal{CR(H)}$ if it satisfies
$$ \mathrm{TT}^{\dag_\mathrm{A}} \mathrm{T}=\mathrm{T},~\mathrm{T}^{\dag_\mathrm{A}} \mathrm{TT}^{\dag_\mathrm{A}}=\mathrm{T}^{\dag_\mathrm{A}},~\mathrm{ATT}^{\dag_\mathrm{A}}=(\mathrm{TT}^{\dag_\mathrm{A}})^*\mathrm{A},~\mathrm{AT}^{\dag_\mathrm{A}} \mathrm{T}=(\mathrm{T}^{\dag_\mathrm{A}} \mathrm{T})^*\mathrm{A}.$$
When $\mathrm{A}=\mathrm{I}$, the $\mathrm{A}$-generalized inverse $\mathrm{T}^{\dag_\mathrm{A}}$ reduces to the standard Moore–Penrose inverse $\mathrm{T}^\dag$. In general, the existence of an $\mathrm{A}$-generalized inverse is not guaranteed for all closed-range operators unless $\mathrm{A}$ is invertible. Given a closed subspace $V$ of $\mathcal{H}$, the pair $(\mathrm{A}, V)$ is said to be compatible if the set
$\mathcal{P}(\mathrm{A}, V):= \{\mathrm{Q} \in\mathcal{B(H)} : \mathrm{Q}^2 = \mathrm{Q},~ \mathcal{R}(\mathrm{Q}) = V,~ \mathrm{AQ} = \mathrm{Q}^*\mathrm{A}\}$ is nonempty. Equivalently, $(\mathrm{A}, V)$ is compatible if and only if $V \oplus V^{\perp_\mathrm{A}} = \mathcal{H}$ where $V^{\perp_\mathrm{A}} = \{\xi\in \mathcal{H} : \langle \xi,\eta\rangle_\mathrm{A}=0 ~\forall~ \eta\in V \}$. For a compatible pair $(\mathrm{A}, V)$, the set $\mathcal{P}(\mathrm{A}, V)$ may contain either a single element or infinitely many elements. In particular, $\mathcal{P}(\mathrm{A}, V)$ has a unique element if and only if $V \cap \mathcal{N}(\mathrm{A}) = \{0\}$. For more details on compatibility, see \cite{Corach-classification}.

Let $X$ be a nonempty set. A reproducing kernel Hilbert space $\mathcal{H}=\mathcal{H}(X)$ is a Hilbert space of complex valued functions on the set $X$ with the property that for every $x\in X$, the corresponding linear evaluation functional on $\mathcal{H}$ given by $\phi\rightarrow\phi(x)$, is continuous (see \cite{paulsen2016introduction}). By Riesz Representation theorem for each $x\in X$, there exists a unique element $k_x\in\mathcal{H}$ such that $\phi(x)=\left\langle\phi,k_x\right\rangle$ for all $\phi\in\mathcal{H}$. The collection $\{k_x:x\in X\}$ is the set of all reproducing kernels of $\mathcal{H}$ and $\{\hat{k}_x=k_x/\|k_x\|:x\in X\}$ is the set of all normalized reproducing kernels of $\mathcal{H}$. 

Let $\mathrm{T}\in\mathcal{B(H)}$ where $\mathcal{H}$ is a reproducing kernel Hilbert space. The Berezin transform of $\mathrm{T}$ (see \cite{berezin1972covariant,berezin1974quantization}) is the function $\Tilde{\mathrm{T}}$ on $X$ defined by 
 \begin{equation*}
  \Tilde{\mathrm{T}}(x)=\left\langle \mathrm{T}\hat{k}_x,\hat{k}_x\right\rangle~\text{for all}~x\in X.    
 \end{equation*}
The Berezin set (or range) corresponding to the operator $\mathrm{T}$ was introduced in \cite{karaev2013Reproducing} as
\begin{equation*}
    \textbf{Ber}(\mathrm{T})=\{\Tilde{\mathrm{T}}(x):x\in X\}.
\end{equation*}
The Berezin number and Berezin norm of $\mathrm{T}$ (see \cite{bakherad2020new,karaev2006berezin}) are denoted by $\textbf{ber}(\mathrm{T})$ and $\|\mathrm{T}\|_\textbf{Ber}$, respectively, and are defined as
\begin{eqnarray*}
    \textbf{ber}(\mathrm{T})&=&\sup\{|\Tilde{\mathrm{T}}(x)|:x\in X\}\\
\mbox{and } \ \|\mathrm{T}\|_\textbf{ber} &=& \sup\left\{\left|\left\langle \mathrm{T}\hat{k}_x,\hat{k}_y\right\rangle\right|:x,y\in X\right\}.
\end{eqnarray*} 
The Berezin number of an operator $\mathrm{T}\in\mathcal{B(H)}$ satisfies the following properties:
\begin{enumerate}
    \item $\mathbf{ber}(\mathrm{T}) \leq \|\mathrm{T}\|_\textbf{ber}\leq\|\mathrm{T}\|$.
    \item $\mathbf{ber}(\alpha \mathrm{T})=|\alpha|\mathbf{ber}(\mathrm{T})$ for all $\alpha\in\mathbb{C}$.
    \item $\mathbf{ber}(\mathrm{T}+\mathrm{S})\leq\mathbf{ber}(\mathrm{T})+\mathbf{ber}(\mathrm{S})$ for all $\mathrm{T,S}\in\mathcal{B(H)}$.
\end{enumerate}
If $\mathrm{T}\in\mathcal{B(H)}^+$, then $\|\mathrm{T}\|_\textbf{ber}=\textbf{ber}(\mathrm{T})$ (see \cite{bhunia2023inequalities}). Over the years, several mathematicians have studied the Berezin number and Berezin norm inequalities of reproducing kernel Hilbert space operators (see \cite{bakherad2019berezin,bakherad2020complete,bhunia2023new,garayev2021inequalities,guesba2023some,hajmohamadi2020improvements,majee2023numerical,sen2022berezin,yamanci2017numerical,yamanci2022further}).

For $\mathrm{T}\in\mathcal{B(H)}$, the $\mathrm{A}$-Berezin symbol of $\mathrm{T}$ is the function $\Tilde{\mathrm{T}}_\mathrm{A}:X\rightarrow \mathbb{C}$ defined by $\Tilde{\mathrm{T}}_\mathrm{A}(x)=\langle \mathrm{T}\hat{k}_x', \hat{k}_x'\rangle_\mathrm{A}$, where $\hat{k}_x'$ denotes the $\mathrm{A}$-normalized reproducing kernel, i.e., $\hat{k}_x'=\frac{k_x}{\|k_x\|_\mathrm{A}}$ with $\|k_x\|_\mathrm{A}\neq 0$. The $\mathrm{A}$-Berezin set, $\mathrm{A}$-Berezin number and $\mathrm{A}$-Berezin norm of $\mathrm{T}$ are defined as follows: 
\begin{eqnarray*}
   \textbf{Ber}_\mathrm{A}(\mathrm{T})&=&\{\langle \mathrm{T}\hat{k}_x',\hat{k}_x'\rangle_\mathrm{A} :x\in X\},\\
   \textbf{ber}_\mathrm{A}(\mathrm{T})&=&\underset{x\in X}{\sup}|\langle \mathrm{T}\hat{k}_x',\hat{k}_x'\rangle_\mathrm{A}|,\\
   \|\mathrm{T}\|_{\textbf{ber}_\mathrm{A}}&=&\underset{x,y\in X}{\sup}|\langle \mathrm{T}\hat{k}_x',\hat{k}_y'\rangle_\mathrm{A}|.
\end{eqnarray*}
For details see \cite{Conde-Berezin,Gurdal-A-Berezin,Qiu26Spectralestimates}. For $\mathrm{A}=\mathrm{I}$, the above notions reduce respectively to the Berezin range, Berezin number, and Berezin norm.

Let $X_1,X_2$ be two nonempty sets and $\mathcal{H}_i=\mathcal{H}(X_i)$ be reproducing kernel Hilbert spaces on $X_i$ for $i = 1, 2$. Consider the direct sum $\mathcal{H}=\mathcal{H}_1\oplus \mathcal{H}_2$, then $\mathcal{H}$ is an reproducing kernel Hilbert space on the nonempty set $X_1\times X_2$. Every operator $\mathrm{T}\in\mathcal{B}(\mathcal{H})$ has a $2\times 2$ operator matrix representation 
\begin{eqnarray*}
    &&\mathrm{T}=\begin{bmatrix}
    \mathrm{P}&\mathrm{Q}\\
    \mathrm{R}&\mathrm{S}
\end{bmatrix};~\mathrm{P}\in\mathcal{B}(\mathcal{H}_1),~\mathrm{Q}\in\mathcal{B}(\mathcal{H}_2,\mathcal{H}_1),~\mathrm{R}\in\mathcal{B}(\mathcal{H}_1,\mathcal{H}_2) ~\text{and}~ \mathrm{S}\in\mathcal{B}(\mathcal{H}_2),\end{eqnarray*}
where $\mathcal{B}(\mathcal{H}_i,\mathcal{H}_j)$ is the collection of all bounded linear operators from $\mathcal{H}_i$ to $\mathcal{H}_j$. Over time, mathematicians have become increasingly interested in determining bounds for the Berezin number for operator matrices (see \cite{positivity of,On the Berezin,FURTHER BEREZIN NUMBER}).

In this article, we derive the $\mathrm{A}$-Berezin number and the $\mathrm{A}$-Berezin norm inequalities as well as the Berezin number and the Berezin norm inequalities for bounded linear operators on reproducing kernel Hilbert spaces using the Moore-Penrose inverse. As an application, we explicitly give some bounds for the Berezin number and the Berezin norm inequalities for $2 \times 2$ 
block matrices using Moore-Penrose inverse. With the help of examples, we show that the resulting inequalities yield sharper bounds than previously established results.

\section*{Prerequisites}
In this section, we present the following lemmas that will be used to establish our
results in this article. Throughout this article, we assume that $\mathrm{A}\in\mathcal{B(H)}^+$.

In \cite{pecaric2005mond}, Pečarić et al. have proved the following result. 
\begin{lemma}\label{(A)^r><A^r}
    If $\mathrm{T}\in\mathcal{B(H)}$ is positive and $u\in\mathcal{H}$ with $\|u\|=1$, then
    \begin{enumerate}
        \item $\left\langle \mathrm{T}u,u\right\rangle^r\leq\left\langle \mathrm{T}^ru,u\right\rangle$ for $r\geq1,$
        \item $\left\langle \mathrm{T}u,u\right\rangle^r\geq\left\langle \mathrm{T}^ru,u\right\rangle$ for $0<r\leq1.$
    \end{enumerate}
\end{lemma}

In \cite{buzano1971generalizzazione}, Buzano has established the following inequality.
\begin{lemma}\label{Buzano}
    If $u,v,w\in\mathcal{H}$ and $\|w\|=1$, then the following inequality holds:
    $$\left|\left\langle u,w\right\rangle\left\langle w,v\right\rangle\right|\leq\frac{1}{2}\left( \left|\left\langle u,v\right\rangle\right|+\left\langle u,u\right\rangle^{\frac{1}{2}}\left\langle v,v\right\rangle^{\frac{1}{2}}\right).$$
\end{lemma}

Later on, Saddi \cite{Saddi A-Normal} has generalized the Buzano's inequality as follows:
\begin{lemma}\label{Gen-Buzano}
For $u,v,w\in\mathcal{H}$ and $\|w\|_\mathrm{A}=1$. Then
    \begin{equation*}
        |\langle u,w\rangle_\mathrm{A}\langle w,v\rangle_\mathrm{A}|\leq\frac{1}{2}(|\langle u,v\rangle_\mathrm{A}|+\|u\|_\mathrm{A}\|v\|_\mathrm{A}).
    \end{equation*}
\end{lemma}

In \cite{Al-Twaijry-Generalized}, Altwaijry et al. have refined the well-known Cauchy-Schwarz inequality, which is given as follows:
\begin{lemma}\label{Gen-Cauchy-Sch}
    Consider $u,v\in\mathcal{H}$ and $\epsilon\in[0,1]$. Then
    \begin{equation*}
        |\langle u,v\rangle_\mathrm{A}|\leq\sqrt{\epsilon\|u\|_\mathrm{A}^2\|v\|_\mathrm{A}^2+(1-\epsilon)|\langle u,v\rangle_\mathrm{A}|\|u\|_\mathrm{A}\|v\|_\mathrm{A}}\leq \|u\|_\mathrm{A}\|v\|_\mathrm{A}.
    \end{equation*}
\end{lemma}

In a recent work, Sababheh et al. \cite{sababheh2024numerical} have provided the following lemma involving the Moore-Penrose inverse of an operator.
\begin{lemma}\label{Moore-Penrose inverse}
    Let $\mathrm{T}\in\mathcal{CR(H)}$ and $u,v\in\mathcal{H}$. Then
    $$\left|\left\langle \mathrm{T}u,v\right\rangle\right|^2\leq \left\langle|\mathrm{T}|^2u,u\right\rangle\left\langle \mathrm{TT}^\dag v,v\right\rangle.$$
\end{lemma}

A characterization of the existence of $\mathrm{A}$-generalized inverses in terms of compatible pairs is provided by Corach et al. in Theorem $3.1$ of \cite{Corach-Weighted} as follows:
\begin{theorem}\label{compatibletheorem}
 $\mathrm{T}\in \mathcal{B(H)}$ with closed range, $\mathrm{T}$ admits an $\mathrm{A}$-generalized inverse if and only if the pairs $(\mathrm{A},\mathcal{R}(\mathrm{T}))$ and $(\mathrm{A},\mathcal{N}(\mathrm{T}))$ are compatible.   
\end{theorem}

\noindent Moreover, the Moore-Penrose inverse of a block matrix is given by Hung and Markham in \cite{hung1975moore} as follows:
\begin{theorem}\label{theorem penrose matrix}
 Let $\mathrm{M}=\begin{bmatrix}
     \mathrm{M}_1&\mathrm{M}_2\\
     \mathrm{M}_3&\mathrm{M}_4
 \end{bmatrix}_{m\times n}$. Then
$\mathrm{M}^\dag=\begin{bmatrix}
     \mathrm{K}^\dag(\mathrm{M}_1^*-\mathrm{EF})&\mathrm{K}^\dag(\mathrm{M}_3^*-\mathrm{EH})\\
     \mathrm{F}&\mathrm{H}
 \end{bmatrix}_{n\times m},$
  where 
  \begin{eqnarray*}
      \mathrm{K}&=&\mathrm{M}_1^*\mathrm{M}_1+\mathrm{M}_3^*\mathrm{M}_3,~\mathrm{E}=\mathrm{M}_1^*\mathrm{M}_2+\mathrm{M}_3^*\mathrm{M}_4,\\
      \mathrm{R}&=&\mathrm{M}_2-\mathrm{M}_1\mathrm{K}^\dag \mathrm{E},~\mathrm{S}=\mathrm{M}_4-\mathrm{M}_3\mathrm{K}^\dag \mathrm{E},\\
      \mathrm{L}&=&\mathrm{R}^*\mathrm{R}+\mathrm{S}^*\mathrm{S},~\mathrm{T}=\mathrm{K}^\dag \mathrm{E}(\mathrm{I}-\mathrm{L}^\dag \mathrm{L}),\\
      \mathrm{F}&=&\mathrm{L}^\dag \mathrm{R}^*+(\mathrm{I}-\mathrm{L}^\dag \mathrm{L})(\mathrm{I}+\mathrm{T}^*\mathrm{T})^{-1}(\mathrm{K}^\dag \mathrm{E})^*\mathrm{K}^\dag(\mathrm{M}_1^*-\mathrm{EL}^\dag \mathrm{R}^*),\\
      \mathrm{H}&=&\mathrm{L}^\dag \mathrm{S}^*+(\mathrm{I}-\mathrm{L}^\dag \mathrm{L})(\mathrm{I}+\mathrm{T}^*\mathrm{T})^{-1}(\mathrm{K}^\dag \mathrm{E})^*\mathrm{K}^\dag(\mathrm{M}_3^*-\mathrm{EL}^\dag \mathrm{S}^*).
  \end{eqnarray*}

\end{theorem}
Now putting $\mathrm{M}_1 = \mathrm{M}_4 = 0$ in Theorem \ref{theorem penrose matrix}, the matrix $\mathrm{M}$ becomes $\begin{bmatrix}
    0&\mathrm{M}_2\\
    \mathrm{M}_3&0
\end{bmatrix}$ and    \begin{equation}\label{moor_pen_block_matr}
           \begin{bmatrix}
    0&\mathrm{M}_2\\
    \mathrm{M}_3&0
\end{bmatrix}^\dag=\begin{bmatrix}
              0&(\mathrm{M}_3^*\mathrm{M}_3)^\dag \mathrm{M}_3^*\\
              (\mathrm{M}_2^*\mathrm{M}_2)^\dag \mathrm{M}_2^*&0
          \end{bmatrix}.  
          \end{equation} 
The following lemmas on the Berezin number and the Berezin norm of operator matrices will be needed to establish our results.
This result is established by M. Bakherad in \cite{bakherad2018some}.
\begin{lemma}\label{matrixber}
    Let $\mathrm{S}\in\mathcal{B}(\mathcal{H}_1),~ \mathrm{P}\in\mathcal{B}(\mathcal{H}_2,\mathcal{H}_1),~ \mathrm{Q}\in\mathcal{B}(\mathcal{H}_1,\mathcal{H}_2)$ and $\mathrm{T}\in\mathcal{B}(\mathcal{H}_2)$. Then the following inequalities hold:
    \begin{eqnarray*}
            \mathbf{ber}\left(\begin{bmatrix}
                \mathrm{S}&0\\
                0&\mathrm{T}
            \end{bmatrix}\right) &\leq& \max\{\mathbf{ber}(\mathrm{S}),\mathbf{ber}(\mathrm{T})\},\\
             \mbox{ and }~ \ \mathbf{ber}\left(\begin{bmatrix}
                0&\mathrm{P}\\
                \mathrm{Q}&0
            \end{bmatrix}\right)&\leq&\frac{1}{2}(\|\mathrm{P}\|+\|\mathrm{Q}\|).
        \end{eqnarray*}
    In particular, if $\mathcal{H}_1=\mathcal{H}_2$, then $\mathbf{ber}\left(\begin{bmatrix}
                0&\mathrm{P}\\
                \mathrm{P}&0
            \end{bmatrix}\right)\leq\|\mathrm{P}\|$.
\end{lemma}

Bhunia et al. \cite{bhunia2024berezin} have proved the following Berezin norm inequality.
\begin{corollary}\label{ber norm lemma}
    Let $\mathrm{S}\in\mathcal{B}(\mathcal{H}_1),~ \mathrm{P}\in\mathcal{B}(\mathcal{H}_2,\mathcal{H}_1),~ \mathrm{Q}\in\mathcal{B}(\mathcal{H}_1,\mathcal{H}_2)$ and $\mathrm{T}\in\mathcal{B}(\mathcal{H}_2)$. Then
    \begin{eqnarray*}
            \left\|\begin{bmatrix}
                \mathrm{S}&0\\
                0&\mathrm{T}
            \end{bmatrix}\right\|_\mathbf{ber} &\leq& \max\{\|\mathrm{S}\|_\mathbf{ber},\|\mathrm{T}\|_\mathbf{ber}\} \\
           \mbox{ and } ~ \left\|\begin{bmatrix}
                0&\mathrm{P}\\
                \mathrm{Q}&0
            \end{bmatrix}\right\|_\mathbf{ber} &\leq& \max\{\|\mathrm{P}\|_\mathbf{ber},\|\mathrm{Q}\|_\mathbf{ber}\}.
            \end{eqnarray*}
\end{corollary}

\section{Upper bound of the \texorpdfstring{$\mathrm{A}$}~-Berezin number of an operator}
In this section, we derive the upper bound of the Berezin number and $\mathrm{A}$-Berezin number of an operator with the help of the Moore-Penrose inverse. The next lemma, which is frequently used for establishing our results, also generalizes Lemma \ref{Moore-Penrose inverse}.
\begin{lemma}\label{Gen-Moore-Pen}
Let $\mathrm{T}\in\mathcal{B}_\mathrm{A}(\mathcal{H})$ has closed range with the pairs $(\mathrm{A},\mathcal{R}(\mathrm{T}))$ and $(\mathrm{A},\mathcal{N}(\mathrm{T}))$ are compatible, then
    \begin{equation*}
    |\langle \mathrm{T}u,v\rangle_\mathrm{A}|^2\leq\langle \mathrm{T}^{\#_\mathrm{A}}\mathrm{T}u,u\rangle_\mathrm{A}\langle \mathrm{TT}^{\dag_\mathrm{A}} v,v\rangle_\mathrm{A}~~\mbox{for}~u,v\in\mathcal{H}.    
    \end{equation*}
\end{lemma}

\begin{proof}
By Theorem \ref{compatibletheorem}, it follows that $\mathrm{T}^{\dag_\mathrm{A}}$ exists. Now
 \begin{eqnarray*}
     |\langle \mathrm{T}u,v\rangle_\mathrm{A}|^2 =|\langle \mathrm{AT}u,v\rangle|^2
     =|\langle \mathrm{ATT}^{\dag_\mathrm{A}} \mathrm{T}u,v\rangle|^2
     &=&|\langle \mathrm{T}u,(\mathrm{ATT}^{\dag_\mathrm{A}})^*v\rangle|^2\\
     &=&|\langle \mathrm{T}u,(\mathrm{TT}^{\dag_\mathrm{A}})^*\mathrm{A}v\rangle|^2\\
     &=&|\langle \mathrm{T}u,\mathrm{ATT}^{\dag_\mathrm{A}} v\rangle|^2\\
     &=&|\langle \mathrm{AT}u,\mathrm{TT}^{\dag_\mathrm{A}} v\rangle|^2\\
     &=&|\langle \mathrm{T}u,\mathrm{TT}^{\dag_\mathrm{A}} v\rangle_\mathrm{A}|^2.
     \end{eqnarray*} 
Using Lemma \ref{Gen-Cauchy-Sch}, we have
\begin{eqnarray*}
     |\langle \mathrm{T}u,v\rangle_\mathrm{A}|^2 &\leq&\langle \mathrm{T}u,\mathrm{T}u\rangle_\mathrm{A}\langle \mathrm{TT}^{\dag_\mathrm{A}} v,\mathrm{TT}^{\dag_\mathrm{A}} v\rangle_\mathrm{A}\\
     &=&\langle \mathrm{T}u,\mathrm{T}u\rangle_\mathrm{A}\langle \mathrm{ATT}^{\dag_\mathrm{A}} v,\mathrm{TT}^{\dag_\mathrm{A}} v\rangle\\
     &=&\langle \mathrm{T}u,\mathrm{T}u\rangle_\mathrm{A}\langle (\mathrm{TT}^{\dag_\mathrm{A}})^*\mathrm{A}v,\mathrm{TT}^{\dag_\mathrm{A}} v\rangle\\
     &=&\langle \mathrm{T}u,\mathrm{T}u\rangle_\mathrm{A}\langle \mathrm{A}v,\mathrm{TT}^{\dag_\mathrm{A}} \mathrm{TT}^{\dag_\mathrm{A}} v\rangle\\
     &=&\langle \mathrm{T}u,\mathrm{T}u\rangle_\mathrm{A}\langle \mathrm{A}v,\mathrm{TT}^{\dag_\mathrm{A}} v\rangle\\
     &=&\langle \mathrm{T}u,\mathrm{T}u\rangle_\mathrm{A}\langle (\mathrm{TT}^{\dag_\mathrm{A}})^*\mathrm{A}v,v\rangle\\
     &=&\langle \mathrm{T}u,\mathrm{T}u\rangle_\mathrm{A}\langle \mathrm{ATT}^{\dag_\mathrm{A}} v,v\rangle\\
     &=&\langle \mathrm{T}u,\mathrm{T}u\rangle_\mathrm{A}\langle \mathrm{TT}^{\dag_\mathrm{A}} v,v\rangle_\mathrm{A}\\
     &=&\langle \mathrm{T}^{\#_\mathrm{A}}\mathrm{T}u,u\rangle_\mathrm{A}\langle \mathrm{TT}^{\dag_\mathrm{A}} v,v\rangle_\mathrm{A}.
 \end{eqnarray*}   
\end{proof}

\begin{theorem}\label{aberfZ}
Let $\mathrm{Z}\in\mathcal{B}_\mathrm{A}(\mathcal{H})$ has closed range with the pairs $(\mathrm{A},\mathcal{R}(\mathrm{Z}))$ and $(\mathrm{A},\mathcal{N}(\mathrm{Z}))$ are compatible. Then for $\epsilon\in[0,1]$, 
        \begin{equation*}
         \mathbf{ber}_\mathrm{A}^2(\mathrm{Z})\leq  \frac{1-\epsilon}{2}\mathbf{ber}_\mathrm{A}^\frac{1}{2}(\mathrm{Z}^{\#_\mathrm{A}}\mathrm{Z})\mathbf{ber}_\mathrm{A}(\mathrm{Z}^{\#_\mathrm{A}}\mathrm{Z}+\mathrm{ZZ}^{\dag_\mathrm{A}})+\epsilon\mathbf{ber}_\mathrm{A}(\mathrm{Z}^{\#_\mathrm{A}}\mathrm{Z}). 
        \end{equation*}
\end{theorem}
\begin{proof}
 Let $\hat{k}_x'$ be a $\mathrm{A}$-normalized reproducing kernel of $\mathcal{H}$. Then by Lemma \ref{Gen-Cauchy-Sch}, we have 
      \begin{eqnarray*}
        |\langle \mathrm{Z}\hat{k}_x',\hat{k}_x'\rangle_\mathrm{A}|^2&\leq&(1-\epsilon)\langle \mathrm{Z}\hat{k}_x',\mathrm{Z}\hat{k}_x'\rangle_\mathrm{A}^\frac{1}{2}\langle \hat{k}_x',\hat{k}_x'\rangle_\mathrm{A}^\frac{1}{2}|\langle \mathrm{Z}\hat{k}_x',
          \hat{k}_x'\rangle_\mathrm{A}|+\epsilon\langle \mathrm{Z}\hat{k}_x',\mathrm{Z}\hat{k}_x'\rangle_\mathrm{A}\langle \hat{k}_x',\hat{k}_x'\rangle_\mathrm{A}\\
          &=&(1-\epsilon)\langle \mathrm{Z}\hat{k}_x',\mathrm{Z}\hat{k}_x'\rangle_\mathrm{A}^\frac{1}{2}|\langle \mathrm{Z}\hat{k}_x',\hat{k}_x'\rangle_\mathrm{A}|+\epsilon\langle \mathrm{Z}\hat{k}_x',\mathrm{Z}\hat{k}_x'\rangle_\mathrm{A}\\
          &\leq&(1-\epsilon)\langle \mathrm{Z}\hat{k}_x',\mathrm{Z}\hat{k}_x'\rangle_\mathrm{A}^\frac{1}{2}\langle \mathrm{Z}^{\#_\mathrm{A}}\mathrm{Z}\hat{k}_x',
          \hat{k}_x'\rangle_\mathrm{A}^\frac{1}{2}\langle \mathrm{ZZ}^{\dag_\mathrm{A}}\hat{k}_x',
          \hat{k}_x'\rangle_\mathrm{A}^\frac{1}{2}\\
          &&+\epsilon\langle \mathrm{Z}\hat{k}_x',\mathrm{Z}\hat{k}_x'\rangle_\mathrm{A}(\text{by Lemma}~\ref{Gen-Moore-Pen})\\
          &\leq&(1-\epsilon)\langle \mathrm{Z}\hat{k}_x',\mathrm{Z}\hat{k}_x'\rangle_\mathrm{A}^\frac{1}{2}\frac{\langle \mathrm{Z}^{\#_\mathrm{A}}\mathrm{Z}\hat{k}_x',
          \hat{k}_x'\rangle_\mathrm{A}+\langle \mathrm{ZZ}^{\dag_\mathrm{A}}\hat{k}_x',
          \hat{k}_x'\rangle_\mathrm{A}}{2}\\ &&+\epsilon\langle \mathrm{Z}\hat{k}_x',\mathrm{Z}\hat{k}_x'\rangle_\mathrm{A} ~ (\text{by the arithmetic-geometric mean inequality})\\
          &=&(1-\epsilon)\langle \mathrm{Z}\hat{k}_x',\mathrm{Z}\hat{k}_x'\rangle_\mathrm{A}^\frac{1}{2}\frac{\langle (\mathrm{Z}^{\#_\mathrm{A}}\mathrm{Z}+\mathrm{ZZ}^{\dag_\mathrm{A}})\hat{k}_x',
          \hat{k}_x'\rangle_\mathrm{A}}{2}+\epsilon\langle \mathrm{Z}\hat{k}_x',\mathrm{Z}\hat{k}_x'\rangle_\mathrm{A}\\
          &=&\frac{1-\epsilon}{2}\langle \mathrm{Z}^{\#_\mathrm{A}}\mathrm{Z}\hat{k}_x',\hat{k}_x'\rangle_\mathrm{A}^\frac{1}{2}\langle (\mathrm{Z}^{\#_\mathrm{A}}\mathrm{Z}+\mathrm{ZZ}^{\dag_\mathrm{A}})\hat{k}_x',
          \hat{k}_x'\rangle_\mathrm{A}+\epsilon\langle \mathrm{Z}^{\#_\mathrm{A}}\mathrm{Z}\hat{k}_x',\hat{k}_x'\rangle_\mathrm{A}\\
          &\leq&\frac{1-\epsilon}{2}\mathbf{ber}_\mathrm{A}^\frac{1}{2}(\mathrm{Z}^{\#_\mathrm{A}}\mathrm{Z})\mathbf{ber}_\mathrm{A}(\mathrm{Z}^{\#_\mathrm{A}}\mathrm{Z}+\mathrm{ZZ}^{\dag_\mathrm{A}})+\epsilon\mathbf{ber}_\mathrm{A}(\mathrm{Z}^{\#_\mathrm{A}}\mathrm{Z}).
          \end{eqnarray*}
 Now taking supremum over all $x \in X$, we get our required result.
\end{proof}

\begin{remark}
In \cite[Corollary $2.5$]{Zamani2by2}, it is given that 
\begin{equation}\label{Hubanber}
\mathbf{ber}_\mathrm{A}^2(\mathrm{Z})\leq \frac{1}{2}\mathbf{ber}_\mathrm{A}(\mathrm{Z}^{\#_\mathrm{A}}\mathrm{Z}+\mathrm{ZZ}^{\#_\mathrm{A}}).  
\end{equation} 
Also, \cite[Theorem 3.5]{Al-Twaijry-Siberian} gives the following inequality
\begin{equation}\label{Siberian}
\mathbf{ber}_\mathrm{A}^2(\mathrm{Z})\leq \frac{1}{2}\|\mathrm{ZZ}^{\#_\mathrm{A}}+\mathrm{Z}^{\#_\mathrm{A}}\mathrm{Z}\|_{\mathbf{ber}_\mathrm{A}}.  
\end{equation}
Furthermore, from \cite[Theorem $3$]{Albeladi-New-Approach}, we have
    \begin{equation}\label{New-Approach}
\mathbf{ber}_\mathrm{A}^2(\mathrm{Z})\leq \frac{1}{2}\mathbf{ber}_\mathrm{A}(\mathrm{Z}^2)+\frac{1}{4}\|\mathrm{ZZ}^{\#_\mathrm{A}}+\mathrm{Z}^{\#_\mathrm{A}}\mathrm{Z}\|_{\mathbf{ber}_\mathrm{A}}.
\end{equation} 
Take $\mathrm{A}=\begin{bmatrix}
        \frac{1}{3}&0\\
        0&\frac{1}{5}
    \end{bmatrix},\ \mathrm{Z}=\begin{bmatrix}
        1&1\\
        0&0
    \end{bmatrix}$. Then from \eqref{Hubanber} (also from \eqref{Siberian}) we get $\mathbf{ber}_\mathrm{A}^2(\mathrm{Z})\leq 1.834$ and inequality \eqref{New-Approach} gives $ \mathbf{ber}_\mathrm{A}^2(\mathrm{Z})\leq 1.4166$,  whereas from Theorem \ref{aberfZ} with $\epsilon=0$, we get $ \mathbf{ber}_\mathrm{A}^2(\mathrm{Z})\leq 1.291$.
\end{remark}

Taking $\mathrm{A}=\mathrm{I}$, we have the following result.

\begin{corollary}\label{ber2T=XY}
Let $\mathrm{Z}\in \mathcal{CR(H)}$. Then for $\epsilon\in[0,1]$,
        \begin{equation*}
         \mathbf{ber}^2(\mathrm{Z})\leq  \frac{1-\epsilon}{2}\mathbf{ber}^\frac{1}{2}\left(|\mathrm{Z}|^2\right)\left\||\mathrm{Z}|^2+\mathrm{ZZ}^\dag\right\|_{\mathbf{ber}}+\epsilon\left\||\mathrm{Z}|^2\right\|_{\mathbf{ber}}. 
        \end{equation*}
\end{corollary}
       
\begin{corollary}\label{ber2XY}
 Let $\mathrm{P,Q}$ be two $n\times n$ matrices and $\epsilon\in[0,1]$. Then 
 \begin{eqnarray*}
    \mathbf{ber}^2\left(\begin{bmatrix}
              0&\mathrm{P}\\
              \mathrm{Q}&0  
    \end{bmatrix}\right)&\leq&\frac{1-\epsilon}{2}\max \left\{\mathbf{ber}^\frac{1}{2}(|\mathrm{P}|^2),\mathbf{ber}^\frac{1}{2}(|\mathrm{Q}|^2)\right\}\\
    &&\max \left\{\mathbf{ber}(|\mathrm{P}|^2+\mathrm{Q}(|\mathrm{Q}|^2)^\dag \mathrm{Q}^*),\mathbf{ber}(|\mathrm{Q}|^2+\mathrm{P}(|\mathrm{P}|^2)^\dag \mathrm{P}^*)\right\}\\
    &&+\epsilon\max\left\{\mathbf{ber}(|\mathrm{P}|^2),\mathbf{ber}(|\mathrm{Q}|^2)\right\}.
 \end{eqnarray*}
\end{corollary}

\begin{proof}
 Let $\mathrm{Z}=\begin{bmatrix}
            0&\mathrm{P}\\
            \mathrm{Q}&0
  \end{bmatrix}$. Then from \eqref{moor_pen_block_matr} and Corollary \ref{ber2T=XY}, we have
 \begin{eqnarray*}
    \mathbf{ber}^2\left(\begin{bmatrix}
      0&\mathrm{P}\\
      \mathrm{Q}&0  
     \end{bmatrix}\right)
    &\leq&\frac{1-\epsilon}{2}\textbf{ber}^{\frac{1}{2}}\left(\begin{bmatrix}
     |\mathrm{Q}|^2&0\\
     0&|\mathrm{P}|^2
    \end{bmatrix}\right)\\
    &&\textbf{ber}\left(\begin{bmatrix}
              |\mathrm{Q}|^2+\mathrm{P}(|\mathrm{P}|^2)^\dag \mathrm{P}^*&0\\
              0&|\mathrm{P}|^2+\mathrm{Q}(|\mathrm{Q}|^2)^\dag \mathrm{Q}^*
          \end{bmatrix}\right)\\
          &&+\epsilon\textbf{ber}\left(\begin{bmatrix}
              |\mathrm{Q}|^2&0\\
              0&|\mathrm{P}|^2
          \end{bmatrix}\right).
      \end{eqnarray*}
      Applying Lemma \ref{matrixber}, we get our required result.
      \end{proof}

\begin{remark}
 In \cite[Corollary 2.6]{bakherad2018some}, it is given that for $\mathrm{Z}=\begin{bmatrix}
     0&\mathrm{P}\\
     \mathrm{Q}&0
 \end{bmatrix}\in\mathcal{B}(\mathcal{H}_1\oplus\mathcal{H}_2)$, $0\leq p\leq1$ and $r\geq 1$,
 \begin{equation}\label{ber^r(T)bakehard}
     \mathbf{ber}^r(\mathrm{Z})\leq 2^{r-2}\mathbf{ber}^\frac{1}{2}\left(|\mathrm{P}|^{2rp}+|\mathrm{Q}^*|^{2r(1-p)}\right)\mathbf{ber}^\frac{1}{2}\left(|\mathrm{Q}|^{2rp}+|\mathrm{P}^*|^{2r(1-p)}\right).
 \end{equation}
 Let $\mathrm{P}=\begin{bmatrix}
     1&1\\
     0&0
 \end{bmatrix}$ and $\mathrm{Q}=\begin{bmatrix}
     1&0\\
     1&0
 \end{bmatrix}$.
 Then for $r=2,p=\frac{1}{2}$,  we get from \eqref{ber^r(T)bakehard}, 
$\mathbf{ber}^2(\mathrm{Z})\leq 2\sqrt{2}=2.8284$ while from Corollary \ref{ber2XY}, it can be seen that 
 $\mathbf{ber}^2(\mathrm{Z})\leq\sqrt{2}+\frac{2}{3}=2.0808$ for $\epsilon=\frac{1}{3}$. Also for $\epsilon=0$, Corollary \ref{ber2XY} gives $\mathbf{ber}^2(\mathrm{Z})\leq\frac{3}{2}\sqrt{2}=2.1213$.
 Thus, the bound in Corollary \ref{ber2XY} is better than the bound in \eqref{ber^r(T)bakehard}.
\end{remark}

The next result gives an estimation for the upper bound of the Berezin number of an operator.

\begin{theorem}\label{ber2r(Z=XY)}
  Let $\mathrm{Z}\in \mathcal{CR(H)}$. Then for $r\geq 1$
        \begin{equation*}
        \mathbf{ber}^{2r}(\mathrm{Z})\leq  \frac{1}{4}\mathbf{ber}\left(|\mathrm{Z}|^{4r}+(\mathrm{ZZ}^\dag)^{2r}\right)+\frac{1}{2}\mathbf{ber}\left((\mathrm{ZZ}^\dag)^{r}|\mathrm{Z}|^{2r}\right).  
        \end{equation*}
\end{theorem}
\begin{proof}
    Let $\hat{k}_x$ be a normalized reproducing kernel of $\mathcal{H}$. Then
      \begin{eqnarray*}
          &&|\langle \mathrm{Z}\hat{k}_x,\hat{k}_x\rangle|^{2r}\\
          &\leq&(\langle |\mathrm{Z}|^2\hat{k}_x,\hat{k}_x\rangle\langle \mathrm{ZZ}^\dag\hat{k}_x,\hat{k}_x\rangle)^r(\text{by Lemma}~\ref{Moore-Penrose inverse})\\
          &=&\langle |\mathrm{Z}|^2\hat{k}_x,\hat{k}_x\rangle^r\langle \mathrm{ZZ}^\dag\hat{k}_x,\hat{k}_x\rangle^r\\
          &\leq&\langle |\mathrm{Z}|^{2r}\hat{k}_x,\hat{k}_x\rangle\langle (\mathrm{ZZ}^\dag)^r\hat{k}_x,\hat{k}_x\rangle~(\text{by Lemma}~\ref{(A)^r><A^r})
          \\
          &=&\langle |\mathrm{Z}|^{2r}\hat{k}_x,\hat{k}_x\rangle\langle \hat{k}_x,(\mathrm{ZZ}^\dag)^r\hat{k}_x\rangle \ (\text{as}~\mathrm{ZZ}^\dag~\text{is self-adjoint})\\
          &\leq&\frac{1}{2}\left(\langle |\mathrm{Z}|^{2r}\hat{k}_x,|\mathrm{Z}|^{2r}\hat{k}_x\rangle^\frac{1}{2}\langle (\mathrm{ZZ}^\dag)^r\hat{k}_x,(\mathrm{ZZ}^\dag)^r\hat{k}_x\rangle^\frac{1}{2}+|\langle |\mathrm{Z}|^{2r}\hat{k}_x,(\mathrm{ZZ}^\dag)^r\hat{k}_x\rangle|\right)
        \end{eqnarray*}
        by Lemma \ref{Buzano}. Now, applying the arithmetic-geometric mean inequality, we get,
        \begin{eqnarray*}
          &&|\langle \mathrm{Z}\hat{k}_x,\hat{k}_x\rangle|^{2r}\\
          &\leq&\frac{1}{4}\left(\langle |\mathrm{Z}|^{2r}\hat{k}_x,|\mathrm{Z}|^{2r}\hat{k}_x\rangle+\langle (\mathrm{ZZ}^\dag)^{r}\hat{k}_x,(\mathrm{ZZ}^\dag)^r\hat{k}_x\rangle\right)+\frac{1}{2}|\langle |\mathrm{Z}|^{2r}\hat{k}_x,(\mathrm{ZZ}^\dag)^r\hat{k}_x\rangle|\\
          &=&\frac{1}{4}\left(\langle |\mathrm{Z}|^{4r}\hat{k}_x,\hat{k}_x\rangle+\langle (\mathrm{ZZ}^\dag)^{2r}\hat{k}_x,\hat{k}_x\rangle\right)+\frac{1}{2}|\langle (\mathrm{ZZ}^\dag)^r|\mathrm{Z}|^{2r}\hat{k}_x,\hat{k}_x\rangle|\\
          &=&\frac{1}{4} \langle (|\mathrm{Z}|^{4r}+(\mathrm{ZZ}^\dag)^{2r})\hat{k}_x,\hat{k}_x\rangle+\frac{1}{2}|\langle (\mathrm{ZZ}^\dag)^{r}|\mathrm{Z}|^{2r}\hat{k}_x,\hat{k}_x\rangle|\\
          &\leq&\frac{1}{4}\mathbf{ber}\left(|\mathrm{Z}|^{4r}+(\mathrm{ZZ}^\dag)^{2r}\right)+\frac{1}{2}\mathbf{ber}\left((\mathrm{ZZ}^\dag)^{r}|\mathrm{Z}|^{2r}\right).
      \end{eqnarray*}
      Taking the supremum over all $x\in X$, we get our desired result.
\end{proof}

\begin{corollary}\label{ber2r-XY}
  Let $\mathrm{P}$ and $\mathrm{Q}$ be $n\times n$ matrices. Then for $r\in\mathbb{N}$,
        \begin{eqnarray*}
            &&\mathbf{ber}^{2r}\left(\begin{bmatrix}
              0&\mathrm{P}\\
              \mathrm{Q}&0  
            \end{bmatrix}\right)\\
            &\leq&\frac{1}{4}\max \{\mathbf{ber}(|\mathrm{Q}|^{4r}+\mathrm{P}(|\mathrm{P}|^2)^\dag \mathrm{P}^*),\mathbf{ber}(|\mathrm{P}|^{4r}+\mathrm{Q}(|\mathrm{Q}|^2)^\dag \mathrm{Q}^*)\}\\
            &&+\frac{1}{2}\max\{\mathbf{ber}(\mathrm{P}(|\mathrm{P}|^2)^\dag \mathrm{P}^*|\mathrm{Q}|^{2r}),\mathbf{ber}(\mathrm{Q}(|\mathrm{Q}|^2)^\dag \mathrm{Q}^*|\mathrm{P}|^{2r})\}.
        \end{eqnarray*}
\end{corollary}
\begin{proof}
Let $\mathrm{Z}=\begin{bmatrix}
            0&\mathrm{P}\\
            \mathrm{Q}&0
        \end{bmatrix}$. Then from Theorem \ref{ber2r(Z=XY)} and Equation \eqref{moor_pen_block_matr}, we have,
        \begin{eqnarray*}
            &&\mathbf{ber}^{2r}\left(\begin{bmatrix}
              0&\mathrm{P}\\
              \mathrm{Q}&0  
            \end{bmatrix}\right)\\
            &\leq&\frac{1}{4}\max \left\{\mathbf{ber}(|\mathrm{Q}|^{4r}+(\mathrm{P}(|\mathrm{P}|^2)^\dag \mathrm{P}^*)^{2r}),\mathbf{ber}(|\mathrm{P}|^{4r}+(\mathrm{Q}(|\mathrm{Q}|^2)^\dag \mathrm{Q}^*)^{2r})\right\}\\
            &&+\frac{1}{2}\max\left\{\mathbf{ber}((\mathrm{P}(|\mathrm{P}|^2)^\dag \mathrm{P}^*)^r|\mathrm{Q}|^{2r}), \mathbf{ber}((\mathrm{Q}(|\mathrm{Q}|^2)^\dag \mathrm{Q}^*)^r|\mathrm{P}|^{2r})\right\}.
        \end{eqnarray*}
    Since $\mathrm{P}(|\mathrm{P}|^2)^\dag \mathrm{P}^*$ and $\mathrm{Q}(|\mathrm{Q}|^2)^\dag \mathrm{Q}^*$ are idempotent, hence the result follows.
\end{proof}

\begin{remark}
Taking $\alpha=1$ in \cite[Theorem 2.16]{bhunia2024berezin}, we have
\begin{eqnarray}
&&\textbf{ber}^2\left(\begin{bmatrix}
              0&\mathrm{P}\\
              \mathrm{Q}&0  
        \end{bmatrix}\right)\nonumber\\
        &\leq&\max \left\{\mathbf{ber}(\mathrm{PQ}),\mathbf{ber}(\mathrm{QP})\right\}\nonumber\\
        &&+\frac{1}{2}\max\left\{\|\mathrm{PP}^*+\mathrm{Q}^*\mathrm{Q}\|_\mathbf{ber},\|\mathrm{P}^*\mathrm{P}+\mathrm{QQ}^*\|_\mathbf{ber}\right\}\label{somdatta-ber2XY}.    
\end{eqnarray}
 for $\mathrm{P,Q}\in\mathcal{B(H)}$. If we take $\mathcal{H}=\mathbb{C}^2, \ \mathrm{P}=\begin{bmatrix}
              1&1\\
              0&0  
            \end{bmatrix}$ and $\mathrm{Q}=\begin{bmatrix}
              1&0\\
              1&0  
            \end{bmatrix}$, then we get from \eqref{somdatta-ber2XY},
            $\textbf{ber}^2\left(\begin{bmatrix}
              0&\mathrm{P}\\
              \mathrm{Q}&0  
            \end{bmatrix}\right)\leq 4$ while for $r=1$, $\textbf{ber}^2\left(\begin{bmatrix}
              0&\mathrm{P}\\
              \mathrm{Q}&0  
            \end{bmatrix}\right)\leq 2.25$ (from Corollary \ref{ber2r-XY}). Hence, for this example, Corollary \ref{ber2r-XY} gives a better bound than the bound \eqref{somdatta-ber2XY}.
\end{remark}

\section{\texorpdfstring{$\mathrm{A}$}~-Berezin number and \texorpdfstring{$\mathrm{A}$}~-Berezin norm of sum of two operators}
This section is devoted to generalize the estimation of the Berezin number, $\mathrm{A}$-Berezin number and $\mathrm{A}$-Berezin norm of the sum of two operators via the Moore-Penrose inverse.

\begin{theorem}\label{berA+}
Let $\mathrm{M,N}\in\mathcal{B}_\mathrm{A}(\mathcal{H})$ has closed ranges, with the pairs $(\mathrm{A},\mathcal{R}(\mathrm{M}))$, $(\mathrm{A},\mathcal{N}(\mathrm{M}))$, $(\mathrm{A},\mathcal{R}(\mathrm{N}))$ and $(\mathrm{A},\mathcal{N}(\mathrm{N}))$ are compatible. Then 
\begin{eqnarray*}
   \mathbf{ber}_\mathrm{A}^2(\mathrm{M}+\mathrm{N})&\leq&\frac{1}{2}\left(\mathbf{ber}_\mathrm{A}^2(\mathrm{M}^{\#_\mathrm{A}}\mathrm{M}+i\mathrm{N}^{\#_\mathrm{A}}\mathrm{N})+\mathbf{ber}_\mathrm{A}^2(\mathrm{MM}^{\dag_\mathrm{A}}+i\mathrm{NN}^{\dag_\mathrm{A}})\right)\\
   &&+\mathbf{ber}_\mathrm{A}(\mathrm{N}^{\#_\mathrm{A}}\mathrm{M})+\sqrt{\mathbf{ber}_\mathrm{A}(\mathrm{M}^{\#_\mathrm{A}}\mathrm{M})\mathbf{ber}_\mathrm{A}(\mathrm{N}^{\#_\mathrm{A}}\mathrm{N})}.  
\end{eqnarray*}
\end{theorem}
\begin{proof}
Let $\hat{k}_x'$ be a $\mathrm{A}$-normalized reproducing kernel of $\mathcal{H}$. Then 
\begin{eqnarray*}
&&|\langle (\mathrm{M}+\mathrm{N})\hat{k}_x',\hat{k}_x' \rangle_\mathrm{A}|^2\\
&\leq&|\langle \mathrm{M}\hat{k}_x',\hat{k}_x' \rangle_\mathrm{A}|^2+|\langle \mathrm{N}\hat{k}_x',\hat{k}_x'\rangle_\mathrm{A}|^2+2|\langle \mathrm{M}\hat{k}_x',\hat{k}_x'\rangle_\mathrm{A}||\langle \mathrm{N}\hat{k}_x',\hat{k}_x'\rangle_\mathrm{A}|\\
&\leq&\langle \mathrm{M}^{\#_\mathrm{A}}\mathrm{M}\hat{k}_x',\hat{k}_x'\rangle_\mathrm{A}\langle \mathrm{MM}^{\dag_\mathrm{A}} \hat{k}_x',\hat{k}_x' \rangle_\mathrm{A}+\langle \mathrm{N}^{\#_\mathrm{A}}\mathrm{N}\hat{k}_x',\hat{k}_x' \rangle_\mathrm{A}\langle \mathrm{NN}^{\dag_\mathrm{A}} \hat{k}_x',\hat{k}_x'\rangle_\mathrm{A}\\
&&+2|\langle \mathrm{M}\hat{k}_x',\hat{k}_x' \rangle_\mathrm{A}||\langle \mathrm{N}\hat{k}_x',\hat{k}_x'\rangle_\mathrm{A}|~ (\text{by Lemma}~\ref{Gen-Moore-Pen})\\
&\leq&\frac{1}{2}\left(\langle \mathrm{M}^{\#_\mathrm{A}}\mathrm{M}\hat{k}_x',\hat{k}_x' \rangle_\mathrm{A}^2+\langle \mathrm{MM}^{\dag_\mathrm{A}} \hat{k}_x',\hat{k}_x' \rangle_\mathrm{A}^2+\langle \mathrm{N}^{\#_\mathrm{A}}\mathrm{N}\hat{k}_x',\hat{k}_x' \rangle_\mathrm{A}^2+\langle \mathrm{NN}^{\dag_\mathrm{A}} \hat{k}_x',\hat{k}_x' \rangle_\mathrm{A}^2\right)\\
&&+2|\langle \mathrm{M}\hat{k}_x',\hat{k}_x' \rangle_\mathrm{A} \langle\hat{k}_x',\mathrm{N}\hat{k}_x' \rangle_\mathrm{A}|~~(\text{by the arithmetic-geometric mean inequality})\\
&\leq&\frac{1}{2}\left(|\langle (\mathrm{M}^{\#_\mathrm{A}}\mathrm{M}+i\mathrm{N}^{\#_\mathrm{A}}\mathrm{N})\hat{k}_x',\hat{k}_x' \rangle_\mathrm{A}|^2+|\langle (\mathrm{MM}^{\dag_\mathrm{A}}+i\mathrm{NN}^{\dag_\mathrm{A}}) \hat{k}_x',\hat{k}_x' \rangle_\mathrm{A}|^2\right)\\
&&+|\langle \mathrm{N}^{\#_\mathrm{A}}\mathrm{M}\hat{k}_x',\hat{k}_x' \rangle_\mathrm{A}|+ \sqrt{\langle \mathrm{M}^{\#_\mathrm{A}}\mathrm{M}\hat{k}_x',\hat{k}_x'\rangle_\mathrm{A}\langle \mathrm{N}^{\#_\mathrm{A}}\mathrm{N}\hat{k}_x',\hat{k}_x'\rangle_\mathrm{A}}~(\text{by Lemma}~ \ref{Gen-Buzano})\\
&\leq&\frac{1}{2}\left(\textbf{ber}_\mathrm{A}^2(\mathrm{M}^{\#_\mathrm{A}}\mathrm{M}+i\mathrm{N}^{\#_\mathrm{A}}\mathrm{N})+\textbf{ber}_\mathrm{A}^2(\mathrm{MM}^{\dag_\mathrm{A}}+i\mathrm{NN}^{\dag_\mathrm{A}})\right)\\
&&+\textbf{ber}_\mathrm{A}(\mathrm{N}^{\#_\mathrm{A}}\mathrm{M})+ \sqrt{\textbf{ber}_\mathrm{A}(\mathrm{M}^{\#_\mathrm{A}}\mathrm{M})\textbf{ber}_\mathrm{A}(\mathrm{N}^{\#_\mathrm{A}}\mathrm{N})}.
\end{eqnarray*}
Now taking supremum over all $x\in X$, we get the required inequality.
\end{proof}

Taking $\mathrm{A}=\mathrm{I}$ in the above theorem, we have the following result.

\begin{corollary}\label{ber2(A+B)}
Let $\mathrm{M},\mathrm{N}\in\mathcal{CR(H)}$. Then 
\begin{eqnarray*}
    \mathbf{ber}^2(\mathrm{M}+\mathrm{N})&\leq&\frac{1}{2}\left(\mathbf{ber}^2(|\mathrm{M}|^2+i|\mathrm{N}|^2)+\mathbf{ber}^2(\mathrm{MM}^\dag+i\mathrm{NN}^\dag)\right)+\mathbf{ber}( \mathrm{N}^*\mathrm{M})\\
    &&+ \sqrt{\|\mathrm{M}^*\mathrm{M}\|_{\mathbf{ber}}\| \mathrm{N}^* \mathrm{N}\|_\mathbf{ber}} \ .
\end{eqnarray*}
\end{corollary}

\begin{remark}
The following examples are provided to show the improvement obtained by our results.
\begin{enumerate}
    \item In Theorem $2.17$ of \cite{sen2025berezin}, it is given that 
    \begin{equation}\label{ber2A.}
        \mathbf{ber}^2(\mathrm{M}+\mathrm{N})\leq\mathbf{ber}(|\mathrm{M}|+i|\mathrm{N}^*|)\mathbf{ber}(|\mathrm{N}|+i|\mathrm{M}^*|)+\frac{1}{2}\left\||\mathrm{M}|+|\mathrm{N}^*|\right\|_\mathbf{ber}\left\||\mathrm{M}^*|+|\mathrm{N}|\right\|_\mathbf{ber}.
    \end{equation}
 If we take $\mathcal{H}=\mathbb{C}^2$ and $\mathrm{M}=\mathrm{N}=\begin{bmatrix}
     1&1\\
     0&0
 \end{bmatrix}$, then the bound in \eqref{ber2A.} gives $\mathbf{ber}^2(\mathrm{M})\leq 1.1875$
while from Corollary \ref{ber2(A+B)} $\textbf{ber}^2(\mathrm{M})\leq 1$. Therefore, in this case, Corollary \ref{ber2(A+B)} provides sharper bounds than inequality \eqref{ber2A.}.
    \item  In \cite[Theorem 3.7]{mahapatra2025berezin}, it is obtained that  
    \begin{equation}\label{saik3.7}
        \textbf{ber}(\mathrm{M}+\mathrm{N})\leq\frac{1}{\sqrt{2}}\textbf{ber}\left(|\mathrm{M}|^2+|\mathrm{N}|^2+i(|\mathrm{MM}^\dag|+|\mathrm{NN}^\dag|)\right).
    \end{equation}
    Consider $\mathrm{M}=\begin{bmatrix}
        1&1\\
        0&0
    \end{bmatrix}$ and $\mathrm{N}=\begin{bmatrix}
        1&0\\
        1&0
    \end{bmatrix}$. Simple computations shows that Corollary \ref{ber2(A+B)} gives  $\textbf{ber}^2(\mathrm{M}+\mathrm{N})\leq 5.539$, whereas \eqref{saik3.7} gives $\textbf{ber}^2(\mathrm{M}+\mathrm{N})\leq 5.625$.
\end{enumerate}
\end{remark}

The next result gives us another estimation of the $A$-Berezin number of sum of two operators.
\begin{theorem}\label{ber_A^2(M+N)}
Let $\mathrm{M,N}\in\mathcal{B}_\mathrm{A}(\mathcal{H})$ have closed ranges, with the pairs $(\mathrm{A},\mathcal{R}(\mathrm{M}))$, $(\mathrm{A},\mathcal{N}(\mathrm{M}))$, $(\mathrm{A},\mathcal{R}(\mathrm{N}))$ and $(\mathrm{A},\mathcal{N}(\mathrm{N}))$ are compatible. Then
\begin{eqnarray*}
\mathbf{ber}_\mathrm{A}^2(\mathrm{M}+\mathrm{N})&\leq&\mathbf{ber}_\mathrm{A}^2(\mathrm{M})+\mathbf{ber}_\mathrm{A}(\mathrm{M})\mathbf{ber}_\mathrm{A}(\mathrm{N}^{\#_\mathrm{A}}\mathrm{N}+\mathrm{NN}^{\dag_\mathrm{A}})\\
&&+\frac{1}{4}\mathbf{ber}_\mathrm{A}(\mathrm{N}^{\#_\mathrm{A}}\mathrm{N}+\mathrm{NN}^{\#_\mathrm{A}})+\frac{1}{2}\mathbf{ber}_\mathrm{A}(\mathrm{N}^2).
\end{eqnarray*}
\end{theorem}
\begin{proof}
Let $\hat{k}_x'$ be a $\mathrm{A}$-normalized reproducing kernel of $\mathcal{H}$. Then 
\begin{eqnarray*}
   &&|\langle (\mathrm{M}+\mathrm{N})\hat{k}_x',\hat{k}_x'\rangle_\mathrm{A}|^2\\
   &\leq&(|\langle \mathrm{M}\hat{k}_x',\hat{k}_x'\rangle_\mathrm{A}|+|\langle \mathrm{N}\hat{k}_x',\hat{k}_x'\rangle_\mathrm{A}|)^2\\ 
   &=&|\langle \mathrm{M}\hat{k}_x',\hat{k}_x'\rangle_\mathrm{A}|^2+2|\langle \mathrm{M}\hat{k}_x',\hat{k}_x'\rangle_\mathrm{A}||\langle \mathrm{N}\hat{k}_x',\hat{k}_x'\rangle_\mathrm{A}|+|\langle \mathrm{N}\hat{k}_x',\hat{k}_x'\rangle_\mathrm{A}|^2\\
  &\leq&|\langle \mathrm{M}\hat{k}_x',\hat{k}_x'\rangle_\mathrm{A}|^2+2|\langle \mathrm{M}\hat{k}_x',\hat{k}_x'\rangle_\mathrm{A}|\sqrt{\langle \mathrm{N}^{\#_\mathrm{A}}\mathrm{N}\hat{k}_x',\hat{k}_x'\rangle_\mathrm{A}\langle \mathrm{NN}^{\dag_\mathrm{A}}\hat{k}_x',\hat{k}_x'\rangle_\mathrm{A}}\\
  &&+|\langle \mathrm{N}\hat{k}_x',\hat{k}_x'\rangle_\mathrm{A}|^2
\end{eqnarray*}
 by using Lemma \ref{Gen-Moore-Pen}. Thus we have 
\begin{eqnarray*}
    &&|\langle (\mathrm{M}+\mathrm{N})\hat{k}_x',\hat{k}_x'\rangle_\mathrm{A}|^2\\
    &\leq&|\langle \mathrm{M}\hat{k}_x',\hat{k}_x'\rangle_\mathrm{A}|^2+2|\langle \mathrm{M}\hat{k}_x',\hat{k}_x'\rangle_\mathrm{A}|\sqrt{\langle \mathrm{N}^{\#_\mathrm{A}}\mathrm{N}\hat{k}_x',\hat{k}_x'\rangle_\mathrm{A}\langle \mathrm{NN}^{\dag_\mathrm{A}}\hat{k}_x',\hat{k}_x'\rangle_\mathrm{A}}\\
    &&+|\langle \mathrm{N}\hat{k}_x',\hat{k}_x'\rangle_\mathrm{A}\langle \hat{k}_x',\mathrm{N}^{\#_\mathrm{A}}\hat{k}_x'\rangle_\mathrm{A}|\\
    &\leq&|\langle \mathrm{M}\hat{k}_x',\hat{k}_x'\rangle_\mathrm{A}|^2+2|\langle \mathrm{M}\hat{k}_x',\hat{k}_x'\rangle_\mathrm{A}|\sqrt{\langle \mathrm{N}^{\#_\mathrm{A}}\mathrm{N}\hat{k}_x',\hat{k}_x'\rangle_\mathrm{A}\langle \mathrm{NN}^{\dag_\mathrm{A}}\hat{k}_x',\hat{k}_x'\rangle_\mathrm{A}}\\
    &&+\frac{1}{2}(\sqrt{\langle \mathrm{N}\hat{k}_x',\mathrm{N}\hat{k}_x'\rangle_\mathrm{A}\langle \mathrm{N}^{\#_\mathrm{A}}\hat{k}_x',\mathrm{N}^{\#_\mathrm{A}}\hat{k}_x'\rangle_\mathrm{A}}+|\langle \mathrm{N}\hat{k}_x',\mathrm{N}^{\#_\mathrm{A}}\hat{k}_x'\rangle_\mathrm{A}|)(\text{by Lemma}~\ref{Gen-Buzano})\\
    &=&|\langle \mathrm{M}\hat{k}_x',\hat{k}_x'\rangle_\mathrm{A}|^2+2|\langle \mathrm{M}\hat{k}_x',\hat{k}_x'\rangle_\mathrm{A}|\sqrt{\langle \mathrm{N}^{\#_\mathrm{A}}\mathrm{N}\hat{k}_x',\hat{k}_x'\rangle_\mathrm{A}\langle \mathrm{NN}^{\dag_\mathrm{A}}\hat{k}_x',\hat{k}_x'\rangle_\mathrm{A}}\\
    &&+\frac{1}{2}(\sqrt{\langle \mathrm{N}^{\#_\mathrm{A}}\mathrm{N}\hat{k}_x',\hat{k}_x'\rangle_\mathrm{A}\langle \mathrm{NN}^{\#_\mathrm{A}}\hat{k}_x',\hat{k}_x'\rangle_\mathrm{A}}+|\langle \mathrm{N}^2\hat{k}_x',\hat{k}_x'\rangle_\mathrm{A}|)
    \end{eqnarray*}
    Using the arithmetic-geometric mean inequality, we have 
    \begin{eqnarray*}
   &&|\langle (\mathrm{M}+\mathrm{N})\hat{k}_x',\hat{k}_x'\rangle_\mathrm{A}|^2\\
   &\leq& |\langle \mathrm{M}\hat{k}_x',\hat{k}_x'\rangle_\mathrm{A}|^2+|\langle \mathrm{M}\hat{k}_x',\hat{k}_x'\rangle_\mathrm{A}|(\langle \mathrm{N}^{\#_\mathrm{A}}\mathrm{N}\hat{k}_x',\hat{k}_x'\rangle_\mathrm{A}+\langle \mathrm{NN}^{\dag_\mathrm{A}}\hat{k}_x',\hat{k}_x'\rangle_\mathrm{A})\\
    &&+\frac{1}{4}(\langle \mathrm{N}^{\#_\mathrm{A}}\mathrm{N}\hat{k}_x',\hat{k}_x'\rangle_\mathrm{A}+\langle \mathrm{NN}^{\#_\mathrm{A}}\hat{k}_x',\hat{k}_x'\rangle_\mathrm{A})+\frac{1}{2}|\langle \mathrm{N}^2\hat{k}_x',\hat{k}_x'\rangle_\mathrm{A}|\\
    &=&|\langle \mathrm{M}\hat{k}_x',\hat{k}_x'\rangle_\mathrm{A}|^2+|\langle \mathrm{M}\hat{k}_x',\hat{k}_x'\rangle_\mathrm{A}|\langle (\mathrm{N}^{\#_\mathrm{A}}\mathrm{N}+\mathrm{NN}^{\dag_\mathrm{A}})\hat{k}_x',\hat{k}_x'\rangle_\mathrm{A}\\
    &&+\frac{1}{4}\langle (\mathrm{N}^{\#_\mathrm{A}}\mathrm{N}+\mathrm{NN}^{\#_\mathrm{A}})\hat{k}_x',\hat{k}_x'\rangle_\mathrm{A}+\frac{1}{2}|\langle \mathrm{N}^2\hat{k}_x',\hat{k}_x'\rangle_\mathrm{A}|\\
    &\leq&\textbf{ber}_\mathrm{A}^2(\mathrm{M})+\textbf{ber}_\mathrm{A}(\mathrm{M})\textbf{ber}_\mathrm{A}(\mathrm{N}^{\#_\mathrm{A}}\mathrm{N}+\mathrm{NN}^{\dag_\mathrm{A}})+\frac{1}{4}\textbf{ber}_\mathrm{A}(\mathrm{N}^{\#_\mathrm{A}}\mathrm{N}+\mathrm{NN}^{\#_\mathrm{A}})\\
    &&+\frac{1}{2}\textbf{ber}_\mathrm{A}(\mathrm{N}^2).
\end{eqnarray*}
Now taking the supremum over all $x \in X$, we get the required result.
\end{proof}

\begin{remark}
In Proposition $2.2$ of \cite{Gurdal-A-Berezin}, it is given that 
\begin{eqnarray}
\mathbf{ber}_\mathrm{A}^2(\mathrm{M}+\mathrm{N})&\leq&\mathbf{ber}_\mathrm{A}^2(\mathrm{M})+\mathbf{ber}_\mathrm{A}^2(\mathrm{N})\nonumber\\
&&+\frac{1}{2}\left(\textbf{ber}_\mathrm{A}(\mathrm{M}^{\#_\mathrm{A}}\mathrm{M}+\mathrm{NN}^{\#_\mathrm{A}})+\textbf{ber}_\mathrm{A}(\mathrm{NM})\right).~~\label{barsan+prop}  
    \end{eqnarray}
    Let $\mathrm{A}=\begin{bmatrix}
        \frac{1}{3}&0\\
        0&\frac{1}{5}
    \end{bmatrix},\ \mathrm{M}=\begin{bmatrix}
        1&0\\
        1&0
    \end{bmatrix},\ \mathrm{N}=\begin{bmatrix}
        1&1\\
        0&0
    \end{bmatrix}$. Then from \eqref{barsan+prop}, we get $\mathbf{ber}_\mathrm{A}^2(\mathrm{M}+\mathrm{N})\leq 5.134$ whereas from Theorem \ref{ber_A^2(M+N)}, we have $\mathbf{ber}_\mathrm{A}^2(\mathrm{M}+\mathrm{N})\leq 4.417$.
\end{remark}

Setting $\mathrm{A}=\mathrm{I}$ in Theorem \ref{ber_A^2(M+N)}, we have the following result.

\begin{corollary}\label{ber2T}
          Let $\mathrm{M},\mathrm{N}\in \mathcal{CR(H)}$. Then
             \begin{equation*}
                 \mathbf{ber}^2(\mathrm{M}+\mathrm{N})\leq\mathbf{ber}^2(\mathrm{M})+\mathbf{ber}(\mathrm{M})\left\||\mathrm{N}|^2+\mathrm{NN}^\dag\right\|_{\mathbf{ber}}+\frac{1}{4}\left\||\mathrm{N}|^2+|\mathrm{N}^*|^2\right\|_{\mathbf{ber}}+\frac{1}{2}\mathbf{ber}(\mathrm{N}^2).
             \end{equation*}
    \end{corollary}

    \begin{remark}
The following examples highlight the refinements obtained from our results.
\begin{enumerate}
    \item In Theorem $3.9$ of \cite{mahapatra2025berezin} it is given that for $\mathrm{M,N}\in \mathcal{CR(H)}$
\begin{eqnarray}
    \textbf{ber}^2(\mathrm{M}+\mathrm{N})&\leq&\textbf{ber}(\mathrm{M}^*\mathrm{M}+i\mathrm{NN}^\dag )\textbf{ber}(\mathrm{N}^*\mathrm{N}+i\mathrm{MM}^\dag)\nonumber\\
    &&+\frac{1}{2}\|\mathrm{M}^*\mathrm{M}+\mathrm{NN}^\dag\|_\textbf{ber}\|\mathrm{N}^*\mathrm{N}+\mathrm{MM}^\dag\|_\textbf{ber}.\label{3.9Saikat}
\end{eqnarray}
    Consider $\mathrm{M}=\begin{bmatrix}
        1&0\\
        1&0
    \end{bmatrix}$ and $\mathrm{N}=\begin{bmatrix}
        1&1\\
        0&0
    \end{bmatrix}$. Then Corollary \ref{ber2T} gives $\textbf{ber}^2(\mathrm{M}+\mathrm{N})\leq 4.25$, whereas \eqref{3.9Saikat} gives $\textbf{ber}^2(\mathrm{M}+\mathrm{N})\leq 4.75$. 
    \item For $p=\frac{1}{2}$, Corollary $2.22$ of \cite{Somdattakp} gives
\begin{eqnarray}
    \textbf{ber}^2(\mathrm{M}+\mathrm{N})&\leq&\||\mathrm{M}+\mathrm{N}|^2+|(\mathrm{M}+\mathrm{N})^*|^2+|\mathrm{M}-\mathrm{N}|^2+|(\mathrm{M}-\mathrm{N})^*|^2\|_\textbf{ber}.\label{2.22Somdatta}
\end{eqnarray}
  If we consider $\mathrm{M}=\begin{bmatrix}
        1&0\\
        1&0
    \end{bmatrix}$ and $\mathrm{N}=\begin{bmatrix}
        1&1\\
        0&0
    \end{bmatrix},$ then simple computations shows that Corollary \ref{ber2T} gives  $\textbf{ber}^2(\mathrm{M}+\mathrm{N})\leq 4.25$, whereas \eqref{2.22Somdatta} gives $\textbf{ber}^2(\mathrm{M}+\mathrm{N})\leq 6$. Therefore, in this case, the bound obtained in Corollary \ref{ber2T} gives a better bound than that of \eqref{2.22Somdatta}.
    \end{enumerate}
    \end{remark}
        
\begin{corollary}\label{ber2ABCD}
         Let $\mathrm{S}, \mathrm{T}, \mathrm{P}$ and $\mathrm{Q}$ be $n\times n$ matrices. Then
          \begin{eqnarray*}
              \mathbf{ber}^2\left(\begin{bmatrix}
              \mathrm{S}&\mathrm{P}\\
              \mathrm{Q}&\mathrm{T}
          \end{bmatrix}\right)&\leq&\max\left\{\mathbf{ber}^2(\mathrm{S}),\mathbf{ber}^2(\mathrm{T})\right\}+\max\left\{\mathbf{ber}(\mathrm{S}),\mathbf{ber}(\mathrm{T})\right\}\\
          &&\max\left\{\mathbf{ber}(|\mathrm{Q}|^2+\mathrm{P}(|\mathrm{P}|^2)^\dag \mathrm{P}^*),\mathbf{ber}(|\mathrm{P}|^2+\mathrm{Q}(|\mathrm{Q}|^2)^\dag \mathrm{Q}^*)\right\}\\
          &&+\frac{1}{4}\max\left\{\left\||\mathrm{Q}|^2+|\mathrm{P}^*|^2\right\|_{\mathbf{ber}},\left\||\mathrm{P}|^2+|\mathrm{Q}^*|^2\right\|_{\mathbf{ber}}\right\}\\
          &&+\frac{1}{2}\max\left\{\mathbf{ber}(\mathrm{PQ}),\mathbf{ber}(\mathrm{QP})\right\}.
          \end{eqnarray*}
      \end{corollary}
      \begin{proof}
          Consider the matrix  $\begin{bmatrix}
              \mathrm{S}&\mathrm{P}\\
              \mathrm{Q}&\mathrm{T}
          \end{bmatrix}=\mathrm{M}+\mathrm{N}$ where $\mathrm{M}=\begin{bmatrix}
              \mathrm{S}&0\\
              0&\mathrm{T}
          \end{bmatrix}$ and $\mathrm{N}=\begin{bmatrix}
              0&\mathrm{P}\\
              \mathrm{Q}&0
          \end{bmatrix}.$
          Then from Corollary \ref{ber2T},
          we get our required result by using equation \eqref{moor_pen_block_matr}.
      \end{proof}

    \begin{remark}
     In Corollary 2.4 of \cite{FURTHER BEREZIN NUMBER}, it is given that 
 \begin{equation}\label{Further}
     \textbf{ber}^4\left(\begin{bmatrix}
         \mathrm{S}&\mathrm{P}\\
         \mathrm{P}&\mathrm{S}
     \end{bmatrix}\right)\leq 8\textbf{ber}^4(\mathrm{S})+3\left\||\mathrm{P}|^4+|\mathrm{P}^*|^4\right\|_\textbf{ber}+3\left\||\mathrm{P}|^2+|\mathrm{P}^*|^2\right\|_\textbf{ber}\textbf{ber}(\mathrm{P}^2).
 \end{equation}
Taking $\mathrm{S}=\begin{bmatrix}
     1&1\\
     0&0
 \end{bmatrix}$ and $\mathrm{P}=\begin{bmatrix}
     1&0\\
     1&0
 \end{bmatrix}$, we obtain   $\textbf{ber}^4\left(\begin{bmatrix}
         \mathrm{S}&\mathrm{P}\\
         \mathrm{P}&\mathrm{S}
     \end{bmatrix}\right)\leq 35$ (from \eqref{Further}).\\
     Now setting $\mathrm{T}=\mathrm{S}$ and $\mathrm{Q}=\mathrm{P}$ in Corollary \ref{ber2ABCD}, we get $\textbf{ber}^2\left(\begin{bmatrix}
         \mathrm{S}&\mathrm{P}\\
         \mathrm{P}&\mathrm{S}
     \end{bmatrix}\right)\leq \frac{19}{4}$. So in this case, Corollary \ref{ber2ABCD} gives a better bound than \eqref{Further}.
\end{remark}

Now we note the following scalar inequality, which is needed to prove our next results on the Berezin number of the sum of two operators.
\begin{equation}\label{a+ib}
    |a+b|\leq\sqrt{2}\left|a+ib\right| ~\text{where ~} a,b\in\mathbb{R}.
\end{equation}

\begin{theorem}\label{ber2r(A+B)}
    Let $\mathrm{M,N}\in\mathcal{CR(H)}$, $r\geq 1$. Then 
    \begin{eqnarray*}
   \mathbf{ber}^{2r}(\mathrm{M}+\mathrm{N})&\leq&2^{2r-3}\left(\mathbf{ber}^2\left(|\mathrm{M}|^{2r}+i(\mathrm{NN}^\dag)^r\right)+\mathbf{ber}^2\left(|\mathrm{N}|^{2r}+i(\mathrm{MM}^\dag)^r\right)\right)\\
   &&+2^{2r-4}\mathbf{ber}^2\left(|\mathrm{M}|^{2r}+(\mathrm{NN}^\dag)^r+i(|\mathrm{N}|^{2r}+(\mathrm{MM}^\dag)^r)\right).   
  \end{eqnarray*}
\end{theorem}

\begin{proof}
Since $f(t)=t^r, ~r\geq1$ is a convex function on $[0,\infty)$, therefore
\begin{equation}\label{convex}
    f(a+b)\leq\frac{1}{2}f(2a)+\frac{1}{2}f(2b)=2^{r-1}(a^r+b^r) ~\text{for}~ a,b\in [0,\infty).\\
\end{equation}
Let $\hat{k}_x$ be a normalized reproducing kernel of $\mathcal{H}$. Then we have 
    \begin{eqnarray}\label{A+B power r}
        &&(|\langle \mathrm{M}\hat{k}_x,\hat{k}_x\rangle|+|\langle \mathrm{N}\hat{k}_x,\hat{k}_x\rangle|)^r\nonumber\\
        &\leq&2^{r-1}\left(|\langle \mathrm{M}\hat{k}_x,\hat{k}_x\rangle|^r+|\langle \mathrm{N}\hat{k}_x,\hat{k}_x\rangle|^r\right)\nonumber\\
        &\leq&2^{r-1}\left(\langle |\mathrm{M}|^2\hat{k}_x,\hat{k}_x\rangle^\frac{r}{2}\langle \mathrm{MM}^\dag \hat{k}_x,\hat{k}_x\rangle^\frac{r}{2}+\langle |\mathrm{N}|^2\hat{k}_x,\hat{k}_x\rangle^\frac{r}{2}\langle \mathrm{NN}^\dag \hat{k}_x,\hat{k}_x\rangle^\frac{r}{2}\right)\nonumber\\
        &&(\text{by Lemma}~\ref{Moore-Penrose inverse})\nonumber\\
        &\leq&2^{r-2}\left(\langle |\mathrm{M}|^2\hat{k}_x,\hat{k}_x\rangle^r+\langle \mathrm{MM}^\dag\hat{k}_x,\hat{k}_x\rangle^r+\langle |\mathrm{N}|^2\hat{k}_x,\hat{k}_x\rangle^r+\langle \mathrm{NN}^\dag\hat{k}_x,\hat{k}_x\rangle^r\right)\nonumber\\
        &&(\text{by the arithmetic-geometric mean inequality})\nonumber\\
        &\leq&2^{r-2}(\langle |\mathrm{M}|^{2r}\hat{k}_x,\hat{k}_x\rangle+\langle (\mathrm{MM}^\dag)^r\hat{k}_x,\hat{k}_x\rangle+\langle |\mathrm{N}|^{2r}\hat{k}_x,\hat{k}_x\rangle\nonumber\\
        &&+\langle (\mathrm{NN}^\dag)^r\hat{k}_x,\hat{k}_x\rangle)
\end{eqnarray}        
        follows from Lemma \ref{(A)^r><A^r}. Thus we have
\begin{eqnarray}        
       &&(|\langle \mathrm{M}\hat{k}_x,\hat{k}_x\rangle|+|\langle \mathrm{N}\hat{k}_x,\hat{k}_x\rangle|)^r\nonumber\\
       &\leq& 2^{r-2}\left(\langle (|\mathrm{M}|^{2r}+(\mathrm{NN}^\dag)^r)\hat{k}_x,\hat{k}_x\rangle+\langle (|\mathrm{N}|^{2r}+(\mathrm{MM}^\dag)^r)\hat{k}_x,\hat{k}_x\rangle\right)\nonumber\\
       &\leq& 2^{r-\frac{3}{2}}\left|\langle (|\mathrm{M}|^{2r}+(\mathrm{NN}^\dag)^r+i(|\mathrm{N}|^{2r}+(\mathrm{MM}^\dag)^r))\hat{k}_x,\hat{k}_x\rangle\right|.~(\mbox{from}~\eqref{a+ib})\label{A+B^r}
    \end{eqnarray}
  Therefore, 
  \begin{eqnarray*}
      &&|\langle (\mathrm{M}+\mathrm{N})\hat{k}_x,\hat{k}_x\rangle|^{2r}\\
      &\leq&(|\langle \mathrm{M}\hat{k}_x,\hat{k}_x\rangle|+|\langle \mathrm{N}\hat{k}_x,\hat{k}_x\rangle|)^{2r}\\
      &\leq&2^{2r-2}(|\langle \mathrm{M}\hat{k}_x,\hat{k}_x\rangle|^r+|\langle \mathrm{N}\hat{k}_x,\hat{k}_x\rangle|^r)^2~(\text{from}~\eqref{convex})\\
      &\leq&2^{2r-2}(|\langle \mathrm{M}\hat{k}_x,\hat{k}_x\rangle|^{2r}+|\langle \mathrm{N}\hat{k}_x,\hat{k}_x\rangle|^{2r}+2|\langle \mathrm{M}\hat{k}_x,\hat{k}_x\rangle|^r|\langle \mathrm{N}\hat{k}_x,\hat{k}_x\rangle|^r)\\
      &\leq&2^{2r-2}(\langle |\mathrm{M}|^2\hat{k}_x,\hat{k}_x\rangle^r\langle \mathrm{MM}^\dag\hat{k}_x,\hat{k}_x\rangle^r+\langle |\mathrm{N}|^2\hat{k}_x,\hat{k}_x\rangle^r\langle \mathrm{NN}^\dag\hat{k}_x,\hat{k}_x\rangle^r)\\
      &&+2^{2r-1}(|\langle \mathrm{M}\hat{k}_x,\hat{k}_x\rangle||\langle \mathrm{N}\hat{k}_x,\hat{k}_x\rangle|)^r
      \end{eqnarray*}
      follows from Lemma \ref{Moore-Penrose inverse}. Now using Lemma \ref{(A)^r><A^r}, we have 
      \begin{eqnarray*}
      &&|\langle (\mathrm{M}+\mathrm{N})\hat{k}_x,\hat{k}_x\rangle|^{2r}\\
      &\leq&2^{2r-2}(\langle |\mathrm{M}|^{2r}\hat{k}_x,\hat{k}_x\rangle\langle (\mathrm{MM}^\dag)^r\hat{k}_x,\hat{k}_x\rangle+\langle |\mathrm{N}|^{2r}\hat{k}_x,\hat{k}_x\rangle\langle (\mathrm{NN}^\dag)^r\hat{k}_x,\hat{k}_x\rangle)\\
      &&+2^{2r-1}(|\langle \mathrm{M}\hat{k}_x,\hat{k}_x\rangle||\langle \mathrm{N}\hat{k}_x,\hat{k}_x\rangle|)^r\\
      &\leq&2^{2r-3}(\langle |\mathrm{M}|^{2r}\hat{k}_x,\hat{k}_x\rangle^2+\langle (\mathrm{MM}^\dag)^r \hat{k}_x,\hat{k}_x\rangle^2+\langle |\mathrm{N}|^{2r}\hat{k}_x,\hat{k}_x\rangle^2+\langle (\mathrm{NN}^\dag)^r \hat{k}_x,\hat{k}_x\rangle^2)\\
      &&+\frac{1}{2}(|\langle \mathrm{M}\hat{k}_x,\hat{k}_x\rangle|+|\langle \mathrm{N}\hat{k}_x,\hat{k}_x\rangle|)^{2r}\\
      &&(\text{by the arithmetic-geometric mean inequality})\\
      &\leq&2^{2r-3}(|\langle(|\mathrm{M}|^{2r}+i(\mathrm{NN}^\dag)^r)\hat{k}_x,\hat{k}_x\rangle|^2+|\langle(|\mathrm{N}|^{2r}+i(\mathrm{MM}^\dag)^r)\hat{k}_x,\hat{k}_x\rangle|^2)\\
      &&+2^{2r-4}|\langle(|\mathrm{M}|^{2r}+(\mathrm{NN}^\dag)^r+i(|\mathrm{N}|^{2r}+(\mathrm{MM}^\dag)^r))\hat{k}_x,\hat{k}_x\rangle|^2~(\text{by}~\eqref{A+B^r})\\
      &\leq&2^{2r-3}\left(\textbf{ber}^2\left(|\mathrm{M}|^{2r}+i(\mathrm{NN}^\dag)^r\right)+\textbf{ber}^2\left(|\mathrm{N}|^{2r}+i(\mathrm{MM}^\dag)^r\right)\right)\\
      &&+2^{2r-4}\textbf{ber}^2\left(|\mathrm{M}|^{2r}+(\mathrm{NN}^\dag)^r+i(|\mathrm{N}|^{2r}+(\mathrm{MM}^\dag)^r)\right).
  \end{eqnarray*}
  So, taking the supremum over all $x\in X$, we get our desired result.
\end{proof}

\begin{remark}
Taking $r=1$ in Theorem $3.13$ of \cite{mahapatra2025berezin}, we get 
\begin{eqnarray}
    \textbf{ber}^2(\mathrm{M}+\mathrm{N})&\leq&\frac{1}{2}\left(\|(\mathrm{M}^\dag \mathrm{M}+\mathrm{N}^*\mathrm{N})^2\|_\textbf{ber}\|(\mathrm{MM}^*+\mathrm{NN}^\dag)^2\|_\textbf{ber}\right)^\frac{1}{2}\nonumber\\
    &&+\frac{1}{2}\textbf{ber}\left((\mathrm{M}^\dag \mathrm{M}+\mathrm{N}^*\mathrm{N})(\mathrm{MM}^*+\mathrm{NN}^\dag)\right)\label{3.13Saikat}
\end{eqnarray}
for $\mathrm{M,N}\in \mathcal{CR(H)}$. If we consider $\mathrm{M}=\begin{bmatrix}
        1&1\\
        0&0
    \end{bmatrix}$ and $\mathrm{N}=\begin{bmatrix}
        \frac{1}{2}&0\\
        \frac{1}{2}&0
    \end{bmatrix},$ then for $r=1$, Theorem \ref{ber2r(A+B)} gives  $\textbf{ber}^2(\mathrm{M}+\mathrm{N})\leq 2.375$, whereas \eqref{3.13Saikat} gives $\textbf{ber}^2(\mathrm{M}+\mathrm{N})\leq 2.8$. 
\end{remark}

\begin{remark}
Taking $\mathrm{M}=\mathrm{N}$ and $r=1$ in Theorem \ref{ber2r(A+B)}, we have
\begin{equation}\label{ber2rA}
      \mathbf{ber}^{2}(\mathrm{M})\leq\frac{1}{4}\mathbf{ber}^2(|\mathrm{M}|^{2}+i\mathrm{MM}^\dag)+\frac{1}{8}\mathbf{ber}^2(|\mathrm{M}|^{2}+\mathrm{MM}^\dag).   
    \end{equation}
As $|\mathrm{M}|^2+\mathrm{MM}^\dag$ is a positive operator, the bound in \eqref{ber2rA} coincides with the bound given in Corollary 3.10 of \cite{mahapatra2025berezin}.
\end{remark}

\begin{theorem}\label{ber_r_gen}
    Let $\mathrm{M,N}\in\mathcal{CR(H)}$. Then for any $r \geq 1$,
    \begin{equation*}
        \mathbf{ber}^r(\mathrm{M}+\mathrm{N})\leq\frac{2^{r-1}}{\sqrt{2}}\mathbf{ber}\left(|\mathrm{M}|^{2r}+|\mathrm{N}|^{2r}+i((\mathrm{MM}^\dag)^r+(\mathrm{NN}^\dag)^r)\right).
    \end{equation*}
\end{theorem}
\begin{proof}
    Let $\hat{k}_x$ be a normalized reproducing kernel of $\mathcal{H}$. Then
    \begin{eqnarray*}
        &&|\langle (\mathrm{M}+\mathrm{N})\hat{k}_x,\hat{k}_x\rangle|^r\\
         &\leq&(|\langle \mathrm{M}\hat{k}_x,\hat{k}_x\rangle|+|\langle \mathrm{N}\hat{k}_x,\hat{k}_x\rangle|)^r\\
        &\leq&2^{r-2}\left(\langle (|\mathrm{M}|^{2r}+|\mathrm{N}|^{2r})\hat{k}_x,\hat{k}_x\rangle+\langle ((\mathrm{MM}^\dag)^r+(\mathrm{NN}^\dag)^r)\hat{k}_x,\hat{k}_x\rangle\right)
         \mbox{(from \eqref{A+B power r})}\\
        &\leq& 2^{r-2}\sqrt{2}\left|\langle (|\mathrm{M}|^{2r}+|\mathrm{N}|^{2r})\hat{k}_x,\hat{k}_x\rangle+i\langle ((\mathrm{MM}^\dag)^r+(\mathrm{NN}^\dag)^r)\hat{k}_x,\hat{k}_x\rangle\right|~(\text{by}~\eqref{a+ib})\\
        &=& 2^{r-2}\sqrt{2}\left|\langle (|\mathrm{M}|^{2r}+|\mathrm{N}|^{2r}+i((\mathrm{MM}^\dag)^r+(\mathrm{NN}^\dag)^r))\hat{k}_x,\hat{k}_x\rangle\right|\\
        &\leq&\frac{2^{r-1}}{\sqrt{2}}\mathbf{ber}\left(|\mathrm{M}|^{2r}+|\mathrm{N}|^{2r}+i((\mathrm{MM}^\dag)^r+(\mathrm{NN}^\dag)^r)\right).
    \end{eqnarray*}
   Taking the supremum over all $x\in X$, we get our required inequality.
\end{proof}
\begin{remark}
    Putting $r=1$ in Theorem \ref{ber_r_gen}, we get Theorem 3.7 of \cite{mahapatra2025berezin}. Thus, our result generalizes Theorem 3.7 of \cite{mahapatra2025berezin}.
\end{remark}
       
Finally, we derive an upper bound for the $\mathrm{A}$-Berezin norm of the sum of two operators using the Moore-Penrose inverse.

\begin{theorem}\label{anorm}
Let $\mathrm{M,N}\in\mathcal{B}_\mathrm{A}(\mathcal{H})$ has closed ranges, with the pairs $(\mathrm{A},\mathcal{R}(\mathrm{M}))$, $(\mathrm{A},\mathcal{N}(\mathrm{M}))$, $(\mathrm{A},\mathcal{R}(\mathrm{N}))$ and $(\mathrm{A},\mathcal{N}(\mathrm{N}))$ are compatible. Then
         \begin{eqnarray*}
    \|\mathrm{M}+\mathrm{N}\|_{\mathbf{ber}_\mathrm{A}}^2&\leq& \mathbf{ber}_\mathrm{A}(\mathrm{M}^{\#_\mathrm{A}}\mathrm{M}+i\mathrm{N}^{\#_\mathrm{A}}\mathrm{N})\mathbf{ber}_\mathrm{A}(\mathrm{MM}^{\dag_\mathrm{A}}+i\mathrm{NN}^{\dag_\mathrm{A}})\\
    &&+\frac{1}{2}\mathbf{ber}_\mathrm{A}(\mathrm{M}^{\#_\mathrm{A}}\mathrm{M}+\mathrm{N}^{\#_\mathrm{A}}\mathrm{N})+\mathbf{ber}_\mathrm{A}(\mathrm{N}^{\#_\mathrm{A}}\mathrm{M}).
\end{eqnarray*}
\end{theorem}
\begin{proof}
Let $\hat{k}_x'$ and $\hat{k}_y'$ be $\mathrm{A}$-normalized reproducing kernels in $\mathcal{H}$. Then
    \begin{eqnarray*}
        &&|\langle (\mathrm{M}+\mathrm{N})\hat{k}_x',\hat{k}_y'\rangle_\mathrm{A}|^2\\
        &=&|\langle \mathrm{M}\hat{k}_x',\hat{k}_y'\rangle_\mathrm{A} + \langle \mathrm{N}\hat{k}_x',\hat{k}_y'\rangle_\mathrm{A}|^2\\
        &\leq&(|\langle \mathrm{M}\hat{k}_x',\hat{k}_y'\rangle_\mathrm{A}|+|\langle \mathrm{N}\hat{k}_x',\hat{k}_y'\rangle_\mathrm{A}|)^2 \\
        &=&|\langle \mathrm{M}\hat{k}_x',\hat{k}_y'\rangle_\mathrm{A}|^2+|\langle \mathrm{N}\hat{k}_x',\hat{k}_y'\rangle_\mathrm{A}|^2+2|\langle \mathrm{M}\hat{k}_x',\hat{k}_y'\rangle_\mathrm{A}||\langle \mathrm{N}\hat{k}_x',\hat{k}_y'\rangle_\mathrm{A}|\\
        &\leq&\langle \mathrm{M}^{\#_\mathrm{A}}\mathrm{M}\hat{k}_x',\hat{k}_x'\rangle_\mathrm{A}\langle\mathrm{MM}^{\dag_\mathrm{A}}\hat{k}_y',\hat{k}_y'\rangle_\mathrm{A}+\langle\mathrm{N}^{\#_\mathrm{A}}\mathrm{N}\hat{k}_x',\hat{k}_x'\rangle_\mathrm{A}\langle \mathrm{NN}^{\dag_\mathrm{A}}\hat{k}_y',\hat{k}_y'\rangle_\mathrm{A}\\
        &&+2|\langle \mathrm{M}\hat{k}_x',\hat{k}_y'\rangle_\mathrm{A}||\langle \hat{k}_y', \mathrm{N}\hat{k}_x'\rangle_\mathrm{A}|(\text{by Lemma}~\ref{Gen-Moore-Pen})\\
        &\leq&\langle \mathrm{M}^{\#_\mathrm{A}}\mathrm{M}\hat{k}_x',\hat{k}_x'\rangle_\mathrm{A}\langle\mathrm{MM}^{\dag_\mathrm{A}}\hat{k}_y',\hat{k}_y'\rangle_\mathrm{A}+\langle\mathrm{N}^{\#_\mathrm{A}}\mathrm{N}\hat{k}_x',\hat{k}_x'\rangle_\mathrm{A}\langle \mathrm{NN}^{\dag_\mathrm{A}}\hat{k}_y',\hat{k}_y'\rangle_\mathrm{A}\\
        &&+\langle \mathrm{M}\hat{k}_x', \mathrm{M}\hat{k}_x'\rangle_\mathrm{A}^\frac{1}{2}\langle\mathrm{N}\hat{k}_x', \mathrm{N}\hat{k}_x'\rangle_\mathrm{A}^\frac{1}{2}+|\langle\mathrm{M}\hat{k}_x', \mathrm{N}\hat{k}_x'\rangle_\mathrm{A}|(\text{by Lemma}~\ref{Gen-Buzano})\\
        &\leq& \sqrt{\langle \mathrm{M}^{\#_\mathrm{A}}\mathrm{M}\hat{k}_x',\hat{k}_x'\rangle_\mathrm{A}^2+\langle \mathrm{N}^{\#_\mathrm{A}}\mathrm{N}\hat{k}_x',\hat{k}_x'\rangle_\mathrm{A}^2}\sqrt{\langle\mathrm{MM}^{\dag_\mathrm{A}}\hat{k}_y',\hat{k}_y'\rangle_\mathrm{A}^2+\langle\mathrm{NN}^{\dag_\mathrm{A}}\hat{k}_y',\hat{k}_y'\rangle_\mathrm{A}^2}\\
        &&+\langle \mathrm{M}^{\#_\mathrm{A}}\mathrm{M}\hat{k}_x', \hat{k}_x'\rangle_\mathrm{A}^\frac{1}{2}\langle \mathrm{N}^{\#_\mathrm{A}}\mathrm{N}\hat{k}_x', \hat{k}_x'\rangle_\mathrm{A}^\frac{1}{2}+|\langle \mathrm{N}^{\#_\mathrm{A}}\mathrm{M}\hat{k}_x', \hat{k}_x'\rangle_\mathrm{A}|
    \end{eqnarray*}
using the inequality $ac+bd\leq\sqrt{a^2+b^2}\sqrt{c^2+d^2}$ for all $a,b,c,d\in\mathbb{R}$. Now, by the arithmetic-geometric mean inequality, we have 
    \begin{eqnarray*}
        &&|\langle (\mathrm{M}+\mathrm{N})\hat{k}_x',\hat{k}_y'\rangle_\mathrm{A}|^2\\
        &\leq& \sqrt{\langle \mathrm{M}^{\#_\mathrm{A}}\mathrm{M} \hat{k}_x',\hat{k}_x'
        \rangle_\mathrm{A}^2+\langle \mathrm{N}^{\#_\mathrm{A}}\mathrm{N} \hat{k}_x',\hat{k}_x'
        \rangle_\mathrm{A}^2}\sqrt{\langle \mathrm{MM}^{\dag_\mathrm{A}} \hat{k}_y',\hat{k}_y'
        \rangle_\mathrm{A}^2+  \langle \mathrm{NN}^{\dag_\mathrm{A}} \hat{k}_y',\hat{k}_y'
        \rangle_\mathrm{A}^2}\\
        &&+\frac{1}{2}(\langle \mathrm{M}^{\#_\mathrm{A}}\mathrm{M}\hat{k}_x', \hat{k}_x'\rangle_\mathrm{A}+\langle \mathrm{N}^{\#_\mathrm{A}}\mathrm{N}\hat{k}_x', \hat{k}_x'\rangle_\mathrm{A})+|\langle \mathrm{N}^{\#_\mathrm{A}}\mathrm{M}\hat{k}_x', \hat{k}_x'\rangle_\mathrm{A}|\\
        &=& |\langle \mathrm{M}^{\#_\mathrm{A}}\mathrm{M} \hat{k}_x',\hat{k}_x'
        \rangle_\mathrm{A}+i\langle \mathrm{N}^{\#_\mathrm{A}}\mathrm{N} \hat{k}_x',\hat{k}_x'
        \rangle_\mathrm{A}||\langle \mathrm{MM}^{\dag_\mathrm{A}} \hat{k}_y',\hat{k}_y'
        \rangle_\mathrm{A}+i\langle \mathrm{NN}^{\dag_\mathrm{A}} \hat{k}_y',\hat{k}_y'
        \rangle_\mathrm{A}|
        \\&&+\frac{1}{2}\langle(\mathrm{M}^{\#_\mathrm{A}}\mathrm{M}+\mathrm{N}^{\#_\mathrm{A}}\mathrm{N})\hat{k}_x',\hat{k}_x'\rangle_\mathrm{A}+|\langle(\mathrm{N}^{\#_\mathrm{A}}\mathrm{M})\hat{k}_x',\hat{k}_x'\rangle_\mathrm{A}|\\
        &\leq&\textbf{ber}_\mathrm{A}(\mathrm{M}^{\#_\mathrm{A}}\mathrm{M}+i\mathrm{N}^{\#_\mathrm{A}}\mathrm{N})\textbf{ber}_\mathrm{A}(\mathrm{MM}^{\dag_\mathrm{A}}+i\mathrm{NN}^{\dag_\mathrm{A}})+\frac{1}{2}\textbf{ber}_\mathrm{A}(\mathrm{M}^{\#_\mathrm{A}}\mathrm{M}+\mathrm{N}^{\#_\mathrm{A}}\mathrm{N})\\
        &&+\textbf{ber}_\mathrm{A}(\mathrm{N}^{\#_\mathrm{A}}\mathrm{M}).
    \end{eqnarray*}
Taking the supremum over all $x,y\in X$, we get our required result.
     \end{proof}
     
\begin{remark}
    In Theorem $4.4$ of \cite{Conde-Berezin}, it is given that 
    \begin{equation*}
     \|\mathrm{M}+\mathrm{N}\|_{\mathbf{ber}_\mathrm{A}}^2\leq \|\mathrm{M}\|_{\mathbf{ber}_\mathrm{A}}^2+\|\mathrm{N}\|_{\mathbf{ber}_\mathrm{A}}^2+\frac{1}{2}\textbf{ber}_\mathrm{A}(\mathrm{M}^{\#_\mathrm{A}}\mathrm{M}+\mathrm{N}^{\#_\mathrm{A}}\mathrm{N})+\textbf{ber}_\mathrm{A}(\mathrm{M}^{\#_\mathrm{A}}\mathrm{N}).   
    \end{equation*}
   Take $\mathrm{A}=\begin{bmatrix}
        \frac{1}{3}&0\\
        0&\frac{1}{5}
    \end{bmatrix},\ \mathrm{M}=\begin{bmatrix}
        1&1\\
        0&0
    \end{bmatrix},\ \mathrm{N}=\begin{bmatrix}
        1&0\\
        1&0
    \end{bmatrix}$. Then the above inequality gives, $\|\mathrm{M}+\mathrm{N}\|_{\mathbf{ber}_\mathrm{A}}^2\leq 4.967$ whereas from Theorem \ref{anorm}, we get $ \|\mathrm{M}+\mathrm{N}\|_{\mathbf{ber}_\mathrm{A}}^2\leq 4.525$.
\end{remark}

Now, taking $\mathrm{A}=\mathrm{I}$, we have the following result.

\begin{corollary}\label{bernorm2}
  Let $\mathrm{M,N}\in\mathcal{CR(H)}$. Then
         \begin{eqnarray*}
    \left\|\mathrm{M}+\mathrm{N}\right\|_\mathbf{ber}^2&\leq& \mathbf{ber}(|\mathrm{M}|^2+i|\mathrm{N}|^2)\mathbf{ber}(\mathrm{MM}^\dag+i\mathrm{NN}^\dag)+\frac{1}{2}\left\||\mathrm{M}|^2+|\mathrm{N}|^2\right\|_{\mathbf{ber}}\\
    &&+\mathbf{ber}(\mathrm{N}^*\mathrm{M}).
\end{eqnarray*}
\end{corollary}

\begin{remark}
    In $(2.8)$ of \cite{bhunia2023inequalities}, it is given that
\begin{eqnarray}\label{bhunia2.8}
    \|\mathrm{M}+\mathrm{N}\|^2_\textbf{ber}&\leq&2~\textbf{ber}(\mathrm{M}^*\mathrm{M}+\mathrm{N}^*\mathrm{N}).
\end{eqnarray}
    Let $\mathrm{M}=\begin{bmatrix}
       1&1\\
        0&0
    \end{bmatrix}$ and $\mathrm{N}=\begin{bmatrix}
         1&0\\
        1&0
    \end{bmatrix}$. Then from \eqref{bhunia2.8}, we get 
    $\|\mathrm{M}+\mathrm{N}\|^2_\textbf{ber}\leq 6$. Also from Corollary \ref{bernorm2}, we obtain $\|\mathrm{M}+\mathrm{N}\|^2_\textbf{ber}\leq 5$. This shows that Corollary \ref{bernorm2} gives a better bound than \eqref{bhunia2.8}.
\end{remark}

\begin{corollary}\label{matrixmorm}
          Let $\mathrm{S, T, P, Q}$ be $n\times n$ matrices. Then
          \begin{eqnarray*}
              &&\left\|\begin{bmatrix}
              \mathrm{S}&\mathrm{P}\\
              \mathrm{Q}&\mathrm{T}
          \end{bmatrix}\right\|^2_\mathbf{ber}\\
          &\leq&\max\left\{\mathbf{ber}(|\mathrm{S}|^2+i|\mathrm{Q}|^2),\mathbf{ber}(|\mathrm{T}|^2+i|\mathrm{P}|^2)\right\}\\&&\max\left\{\mathbf{ber}(\mathrm{S}(|\mathrm{S}|^2)^\dag \mathrm{S}^*+i \mathrm{P}(|\mathrm{P}|^2)^\dag \mathrm{P}^*),\mathbf{ber}(\mathrm{T}(|\mathrm{T}|^2)^\dag \mathrm{T}^*+i\mathrm{Q}(|\mathrm{Q}|^2)^\dag \mathrm{Q}^*)\right\}\\
          &&+\frac{1}{2}\max\left\{\left\||\mathrm{S}|^2+|\mathrm{Q}|^2\right\|_{\mathbf{ber}},\left\||\mathrm{P}|^2+|\mathrm{T}|^2\right\|_{\mathbf{ber}}\right\}+\frac{1}{2}(\|\mathrm{Q}^*\mathrm{T}\|+\|\mathrm{P}^*\mathrm{S}\|).
          \end{eqnarray*}
      \end{corollary}
      \begin{proof}
          Let $\begin{bmatrix}
              \mathrm{S}&\mathrm{P}\\
              \mathrm{Q}&\mathrm{T}
          \end{bmatrix}=\mathrm{M}+\mathrm{N}$, where $\mathrm{M}=\begin{bmatrix}
              \mathrm{S}&0\\
              0&\mathrm{T}
          \end{bmatrix}$ and $\mathrm{N}=\begin{bmatrix}
              0&\mathrm{P}\\
              \mathrm{Q}&0
          \end{bmatrix}$. Then, using Corollaries \ref{ber norm lemma} and \ref{bernorm2}, Equation \eqref{moor_pen_block_matr} and Lemma \ref{matrixber}, we obtain the desired result.
      \end{proof}

\begin{remark}
In Theorem $2.22$ of \cite{bhunia2024berezin} it is given that 
 \begin{eqnarray}
   \left\|\begin{bmatrix}
         \mathrm{S}&\mathrm{P}\\
         \mathrm{Q}&\mathrm{T}
     \end{bmatrix}\right\|_\textbf{ber}^2&\leq&\max\{\textbf{ber}(|\mathrm{S}|+i|\mathrm{Q}|),\textbf{ber}(|\mathrm{T}|+i|\mathrm{P}|)\}\nonumber\\
     &&\max\{\textbf{ber}(|\mathrm{S}^*|+i|\mathrm{P}^*|),\textbf{ber}(|\mathrm{T}^*|+i|\mathrm{Q}^*|)\}\nonumber\\
     &&+\frac{1}{2}\max\{\||\mathrm{S}|^2+|\mathrm{Q}|^2\|_\textbf{ber},\||\mathrm{P}|^2+|\mathrm{T}|^2\|_\textbf{ber}\}\nonumber\\
     &&+\max\{\|\mathrm{Q}^*\mathrm{T}\|_\textbf{ber},\|\mathrm{P}^*\mathrm{S}\|_\textbf{ber}\}.\label{Normmatrix}
\end{eqnarray} 

 Now taking $\mathrm{S}=\mathrm{T}=\begin{bmatrix}
     \frac{1}{2}&0\\
     \frac{1}{2}&0
 \end{bmatrix}$ and $\mathrm{P}=\mathrm{Q}=\begin{bmatrix}
     1&1\\
     0&0
 \end{bmatrix}$, we get from \eqref{Normmatrix}, $\left\|\begin{bmatrix}
         \mathrm{S}&\mathrm{P}\\
         \mathrm{P}&\mathrm{S}
     \end{bmatrix}\right\|_\textbf{ber}^2\leq 2.7077$ whereas from Corollary \ref{matrixmorm} we get  $\left\|\begin{bmatrix}
         \mathrm{S}&\mathrm{P}\\
         \mathrm{P}&\mathrm{S}
     \end{bmatrix}\right\|_\textbf{ber}^2\leq 2.7071$. So in this case, Corollary \ref{matrixmorm} gives a better bound than that of \eqref{Normmatrix}.
\end{remark}

\textit{Acknowledgements.} 
Mr. Sumon Ghosh would like to thank UGC, Govt. of India, for the financial support (NTA Ref. No. 191620179494) in the form of a fellowship. Dr. S.  Saha Mondal was supported by Indian Institute of Engineering Science and Technology, Shibpur, West Bengal, India under the mentorship
of Dr. S. Ojha during the preparation and initial submission of the article. The authors gratefully acknowledge the reviewers for their constructive comments and suggestions, which have significantly enhanced the quality of the article.

\end{document}